\documentclass[smallextended,referee,envcountsect]{svjour3}
\smartqed

\usepackage{cite}
\usepackage{amsfonts}
\usepackage{amsmath}
\usepackage{amssymb}
\usepackage{graphicx} 
\usepackage{subfig}
\usepackage{caption}
\usepackage{float}
\makeatletter 
\def\smallunderbrace#1{\mathop{\vtop{\m@th\ialign{##\crcr  
$\hfil\displaystyle{#1}\hfil$\crcr 
\noalign{\kern3\p@\nointerlineskip}%
\tiny\upbracefill\crcr\noalign{\kern3\p@}}}}\limits} 
\begin{document}
\title{Error analysis of finite difference/collocation method for the nonlinear coupled parabolic free boundary problem modeling plaque growth in the artery}
\author{F. Nasresfahani\and  M. R. Eslahchi  }
\institute{Mohammad Reza Eslahchi, Corresponding author \\
eslahchi@modares.ac.ir\at
Farzaneh nasresfahani  \\
 f.nasresfahani@modares.ac.ir\\
Department of Applied Mathematics, Faculty of Mathematical Sciences, Tarbiat Modares University, P.O. Box 14115-134 \\
Tehran, Iran
}
\date{}
\maketitle
\begin{abstract}
The main target of this paper is to present a new and efficient method to solve a nonlinear free boundary mathematical model of atherosclerosis. This model consists of three parabolics, one elliptic and one ordinary differential equations that are coupled together and describe the growth of a plaque in the artery. We start our discussion by using the front fixing method to fix the free domain and simplify the model by changing the mix
boundary condition to a Neumann one by applying suitable changes of variables. Then, after employing a nonclassical finite difference and the collocation method on this model, we prove the stability and convergence of methods. Finally, some numerical results are considered to show the efficiency of the method.
\end{abstract}
\section{Introduction}
There are many phenomena that we have questions about and want to describe them or their behaviours. To find an answer we collect data and multiple sources of data support a rapid knowledge. However, our ability to analyze and interpret this data lags far behind data generation and storage capacity. The question that arises here is how can mathematics help us to solve this problem? The answer to this question is that mathematics helps us in this issue by modeling as described in Figure \ref{dra}\cite{meerschaert2013mathematical}.
The first recognizable models were numbers since about 30.000 BC \cite{schichl2004models}. The development of mathematical models continued in various fields. By the invention of calculus by Newton and Leibniz in 1671, differential equations came into existence \cite{newton1774methodus}, and then, partial differential equations stand out in 1719 by Nicolaus Bernoulli \cite{cajori1928early}. 
Early 16th century was the beginning of the modern interaction between mathematics and biology duo to the work of William Harvey \cite{harvey1737exercitatio}. Thereafter, until the 20th century, biological modeling continued to advance by the work of Hardy in 1908 in population genetics \cite{hardy1908mendelian}, Yula in 1925 in birth and death process \cite{yule1925growth}, Luria and Delbruck in 1943 in estimating bacterial mutation rates \cite{luria1943mutations}. The question that matters here is what biological issues  to biological scientists are more valuable and selected for modeling. In reply to the question, according to the authors' consideration, topics of the greatest importance to them are those that affect the human community more, which is the factor that causes a tendency to model  for a more detailed examination of these issues. Due to human society conditions, studying diseases is the most important area of study. According to the ICD 10\footnote{International Classification of Diseases}\cite{world2004international}, the most important diseases that cause of death in the world are HIV\cite{marinho2012model,culshaw2003mathematical}, Tumor \cite{rockne2009mathematical,owen1999mathematical,holmes2000mathematical,chaplain1995mathematical},  Cancer \cite{stephanou2006mathematical,deng2004mathematical,anderson2000mathematical,franks2003modelling}, Cardiovascular diseases (especially Atherosclerosis) \cite{yang2016mathematical,eberhard2006numerical,islam2016mathematical,calvez2009mathematical,calvez2010mathematical,hao2014ldl,2015mathematical} and Wound healing \cite{flegg2012wound,terry2012mathematical,paul2013mathematical,javierre2009mathematical} respectively.
As described above, the heart attack or stroke that happens because of atherosclerosis diseases is one of the third leading cause of death in the world \cite{world2004international}.  There are two various perspectives to research in this area, "modeling" and "numerical analysis" point of view. In the case of modeling, the important point to note here is that  there are different perspectives in studying Atherosclerosis \cite{seo2007mice,ougrinovskaia2010ode}. A very important perspective in the area of studying this disease is to study mathematical modeling in the forms of ordinary differential equation (ODE) and partial differential equation (PDE), that are distinguished from each other by the various factors such as choosing the region and the boundary, elements involved in biologically issues, boundary conditions and etc. For instance, in 2009  a 3-D nonlinear parabolic system of PDEs is considered as a mathematical model of Atherosclerosis involving the local blood flow dynamics by Calvez\cite{calvez2009mathematical}. In 2010 Calvez extended his previous work \cite{calvez2010mathematical}. In both articles, Calvez assumed the artery to be an irregular 3-D cylinder. Friedman in 2014, 2015 presented mathematical models of plaque growth in the artery, which parabolic nonlinear 2-D system of PDEs  with mixed boundary conditions and free boundary are considered for modeling \cite{friedman2015free,hao2014ldl}. It is worth noting that the difference between these models is that in \cite{friedman2015free} the artery is assumed to be a very long circular cylinder but in \cite{hao2014ldl} the artery is considered as an irregular cylinder. In other work, a mathematical model with the approach used in \cite{hao2014ldl} by including the effect of reverse cholesterol transport of plaque growth which includes the (LDL, HDL) concentrations is developed by Friedman \cite{2015mathematical}. There are many other models that are associated with this disease \cite{eberhard2006numerical,little2009model,yang2016mathematical}. In the case of numerical analysis, there are various mathematical methods and different perspectives in convergence and stability analyzing. For instance, in \cite{esmaili2018application} mathematical modeling of a tumor is considered and numerical results with convergence and stability of numerical methods have been presented. In \cite{esmaili2017optimal} mathematical modeling of optimal control of tumor with drug application has been presented and in \cite{esmaili2018application} it has been numerically solved and analyzed. Additionally, in \cite{zhao2017convergence}, a two-dimensional multi-term time fractional diffusion equation is solved numerically using a fully-discrete
schem and convergence and superconvergence of the method is illustrated. For the mathematical analysis point of view, in \cite{ramezani2008combined} a hyperbolic equation with an integral condition is solved using finite difference/spectral method.\\
In this article, we want to solve a free boundary nonlinear system of coupled PDEs that model Atherosclerosis and  consists of three parabolics, one elliptic and one ordinary differential equation which is introduced in \cite{friedman2015free}. For the readers' convenience, we highlight the main goals of this study as follows
\begin{itemize}
\item
We have fixed the domain using the front fixing method and  simplified the model by changing the mix boundary condition  to a Neumann one by applying a suitable change of variables to achieve more comfortable results for numerical analysis.
\item
Applying the finite difference method, we have constructed a sequence, which converges to the exact solution of coupled partial differential equations. 
\item
In each time step, using Taylor theorem,  the problem has changed to linear one (see  (\ref{fin1})-(\ref{fin2})) and using the collocation method, equations (\ref{fin1})-(\ref{fin2}) are solved numerically.
\item
We have proved constructed sequence converges to the exact solution of the problem (see Theorem \ref{convergetheo}) and also the stability of the method is proven (see Theorem \ref{stabtheo}).
\item
We have simulated the model using finite difference and collocation method for some pair of values $(L_0, H_0)$ to show the validity and efficiency of the presented method. It is critical to note that, construction of a new second-order non-classical discretization formula helps us to prove the convergence and stability Theorem appropriately.
\end{itemize}
\begin{figure}
\centering
\includegraphics[scale=.4]{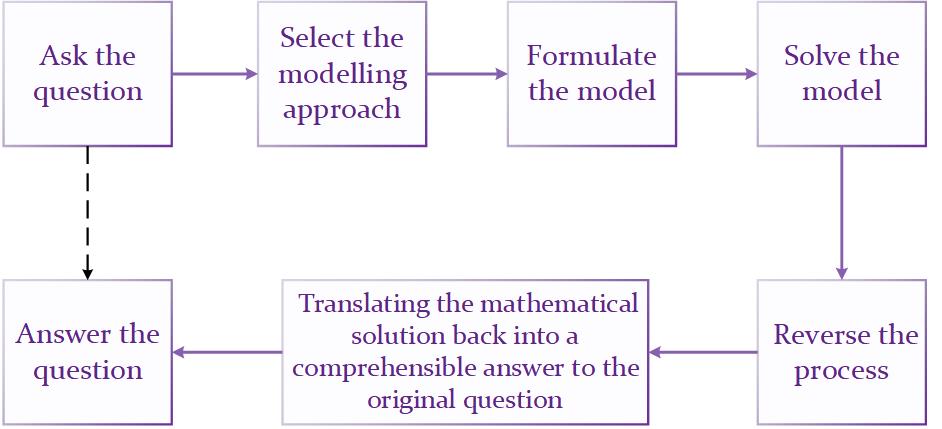}
\caption{\it{Seven steps to show the role of  mathematical modeling  in Solving the mathematical Problems \cite{meerschaert2013mathematical}.}}
\label{dra}
\end{figure}
We organize our paper as follows: In Section \ref{sec2}, we introduce the model of Atherosclerosis presented by Friedman in \cite{friedman2015free}. We apply some changes to construct a more appropriate model for numerical and proof purposes in Section \ref{reform}.
In Section \ref{sec3}, we use finite difference and collocation method with convenience basis for approximating the solution of the problem.
We discuss the stability and convergence of the method used for solving the model in Section \ref{sec4}.
Finally, by presenting some numerical  results, the theoretical statements are justified  in Section \ref{sec5}.
\section{Mathematical model}\label{sec2}
In this paper, we consider the following parabolic free boundary problem modeling plaque growth introduced in \cite{friedman2015free} as follows
\begin{eqnarray}\nonumber
&\dfrac{\partial{\widehat{L}}}{\partial t}-\Delta \widehat{L}=-k_1 \dfrac{(M0-\widehat{F})\widehat{L}}{K_1+\widehat{L}}-r_1\widehat{L},\,\, R(t)<r<1, \, t>0,&\\ 
&\dfrac{\partial \widehat{L}}{\partial n}+\alpha (\widehat{L}-L_0)=0,  \,\, \text{at}\,\, r=R(t), \, t>0 , \quad\dfrac{\partial \widehat{L}}{\partial n}=0,\,\, \text{at}\,\, r=1, \, t>0, \quad \widehat{L}(r ,0)=L_0,& \label{Lb}
\end{eqnarray}
\begin{eqnarray}\nonumber
&\dfrac{\partial \widehat{H}}{\partial t}-\Delta \widehat{H}=-k_2 \dfrac{\widehat{H}\widehat{F}}{K_2+\widehat{F}}-r_2\widehat{H},\,\, R(t)<r<1, \, t>0,&\\ 
&\dfrac{\partial \widehat{H}}{\partial n}+\alpha (\widehat{H}-H_0)=0 ,  \,\, \text{at}\,\, r=R(t), \, t>0,\quad\dfrac{\partial \widehat{H}}{\partial n}=0 ,  \,\, \text{at}\,\, r=1, \, t>0,\quad\widehat{H}(r ,0)=H_0,& \label{Hb}
 \end{eqnarray}
\begin{eqnarray}\nonumber
&\quad\dfrac{\partial \widehat{F}}{\partial t}-D\Delta \widehat{F}+\widehat{F}_r. v=k_1 \dfrac{(M_0-\widehat{F})\widehat{L}}{K1+\widehat{L}}-k_2\dfrac{\widehat{H}\widehat{F}}{K_2+\widehat{F}}-&\\\nonumber
&\lambda \dfrac {\widehat{F}(M_0-\widehat{F})\widehat{L}}{M_0(\delta +\widehat{H})}+\dfrac{\mu_1}{M_0}(M_0-\widehat{F})\widehat{F}-\dfrac{\mu_2}{M_0}\widehat{F}(M_0-\widehat{F}),\,\, R(t)<r<1, \, t>0,&\\
&\dfrac{\partial \widehat{F}}{\partial n}+\beta \widehat{F}=0 \quad ,  \,\, \text{at}\,\, r=R(t), \, t>0,\quad\dfrac{\partial \widehat{F}}{\partial n}=0\quad ,  \,\, \text{at}\,\, r=1, \, t>0,\quad F(r ,0)=0, & \label{Fb}
\end{eqnarray}
\begin{eqnarray}\nonumber
&M_0.v_r=\lambda\dfrac {(M_0-\widehat{F})\widehat{L}}{\delta +\widehat{H}}-\mu_1 (M_0-\widehat{F})-\mu_2 \widehat{F} ,\,\, R(t)<r<1, \, t>0, &\\ \label{vb}
&v(r,t)=0 ,  \,\, \text{at}\,\, r=1, \, t>0,& 
\end{eqnarray}
\begin{eqnarray}\nonumber
&\dfrac{dR(t)}{dt}=v(R(t),t), \, t>0,&\\
&R(0)=\epsilon . & \label{etf}
\end{eqnarray}
There are several mathematical models that describe the growth of a plaque in the artery.
All these models recognize the critical role of the “bad” cholesterols and the “good” cholesterols in determining whether a plaque, once formed, will grow or shrink.\\
In mathematical models, choosing the type of geometry of the problem is very important. Since the plaques grown in the artery are approximately spherically symmetric, in most of the models, it is assumed that the plaque grows radially-symmetric. In this model, it is assumed that the artery is a very long circular cylinder and  a circular cross-section $0\leq r\leq 1$ is considered. The plaque is given by
$R(t) < r < 1$, where $r$ is measured in unit of cm, and $t$ is measured in unit of days. Also, The variables $\widehat{L}$, $\widehat{H}$ and $\widehat{F}$ are taken to be functions of $(r,t)$ only in the region $\{(r,t); R(t) < r < 1, t > 0\}$. In this model, $\widehat{L}$, $\widehat{H}$ and $\widehat{F}$ are the variables that illustrate the concentration of LDL and HDL and the density of foam cell in the plaque respectively and $v$ is the radial velocity. $k_1$ is the rate of ox-LDL ingestion by macrophages. $k_2$ is the rate of reverse cholesterol transport. $r_1$ and $r_2$ 
represent the degradation of the LDL and HDL caused by radicals respectively. $\mu_1$ and $\mu_2$ are the death rate of macrophages and foam cells respectively. Also, D is the diffusion coefficients of foam cells.
 Initially, in order to better understand the model, it is better to know a little how the plaque is calcified in the artery.
 \begin{figure}[!ht]
\centering
\subfloat[{\tiny Intravascular contents}]{\includegraphics[width=3.75cm]{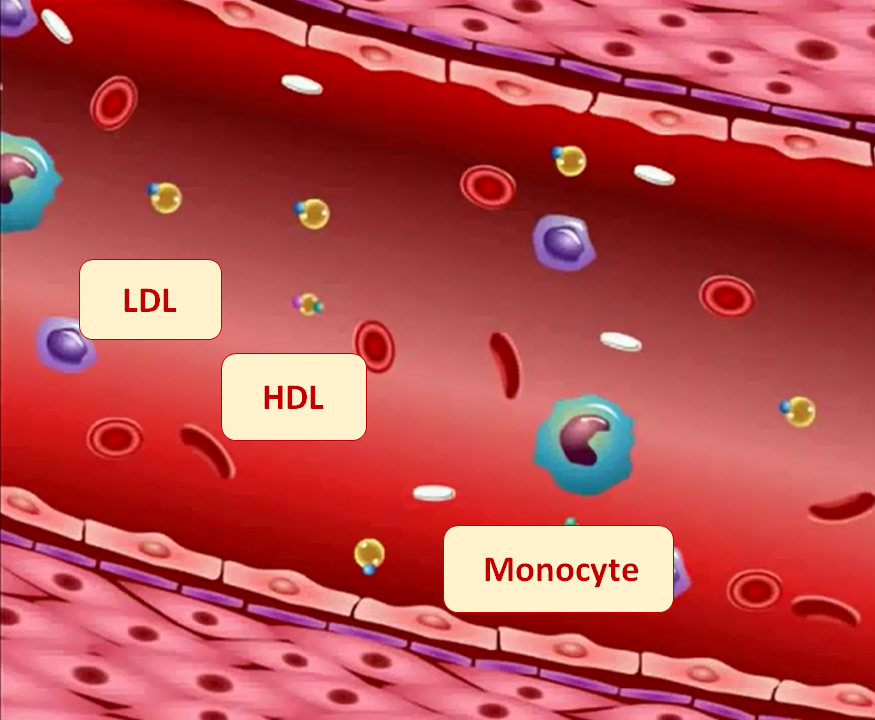}}
\quad
\subfloat[{\tiny LDL penetration}]{\includegraphics[width=3.75cm]{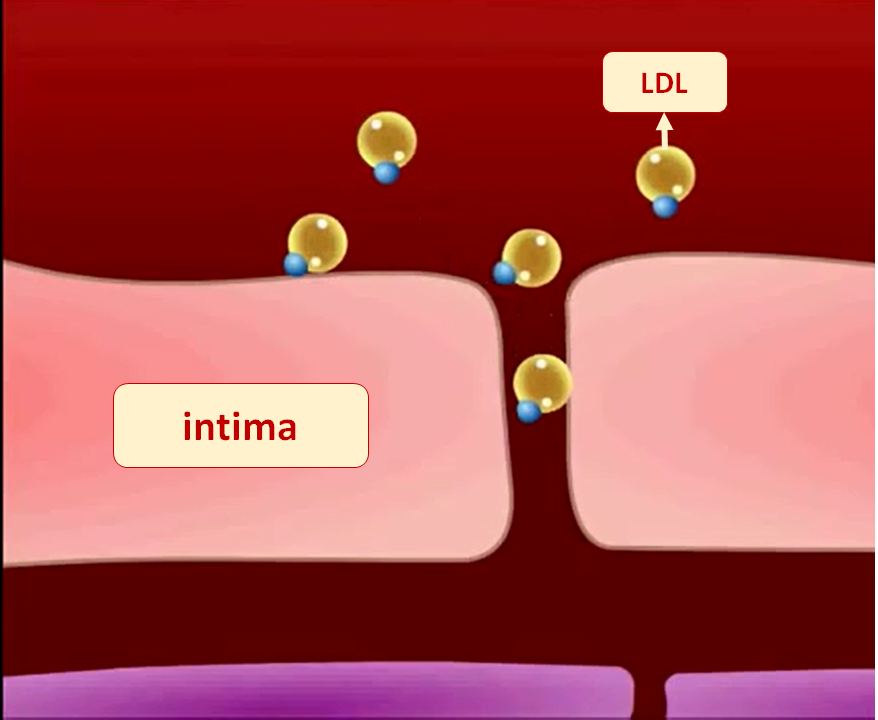}}
\quad
\subfloat[{\tiny FR release}]{\includegraphics[width=3.75cm]{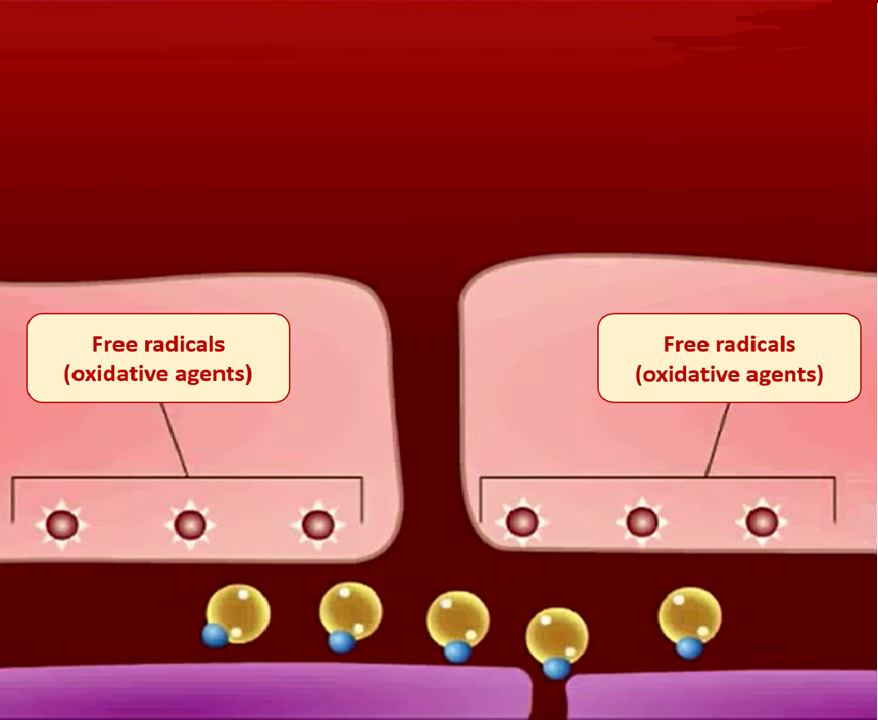}}
\quad
\subfloat[{\tiny L$_{ox}$ formation}]{\includegraphics[width=3.75cm]{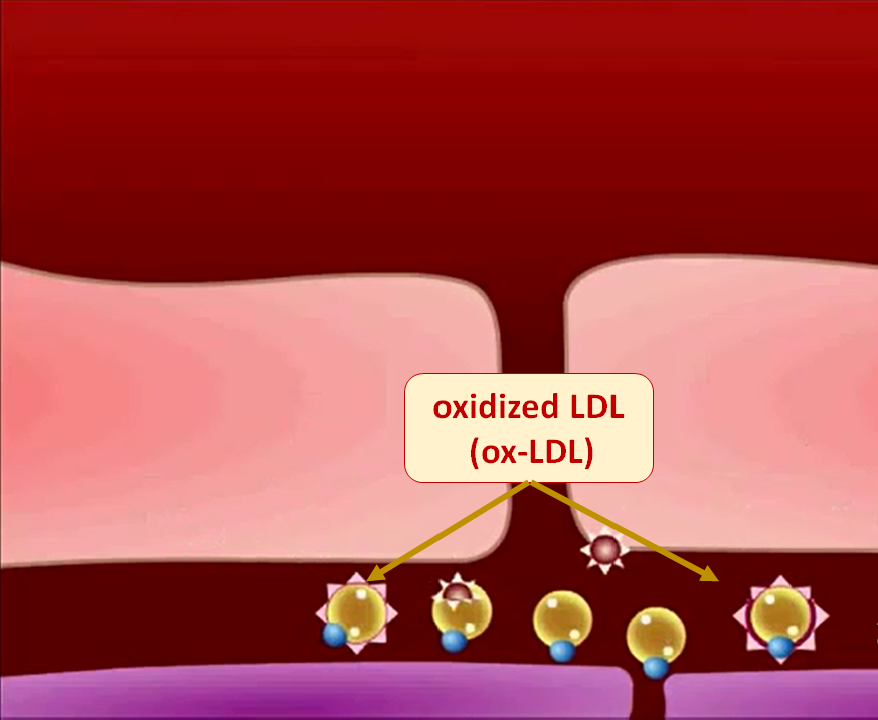}}
\quad
\subfloat[{\tiny M secretion }]{\includegraphics[width=3.75cm]{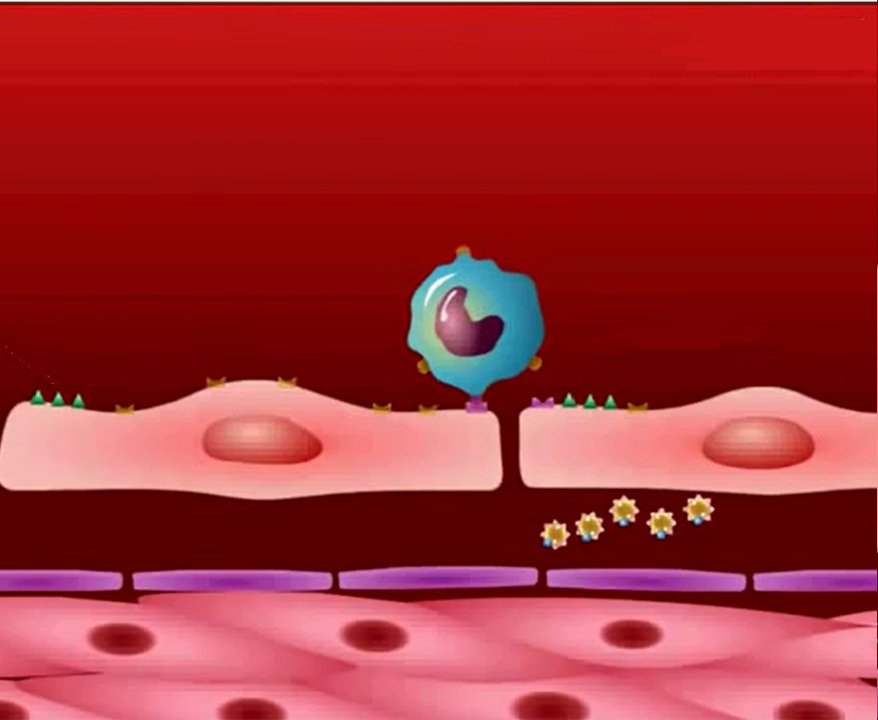}}
\quad
\subfloat[{\tiny M penetration} ]{\includegraphics[width=3.75cm]{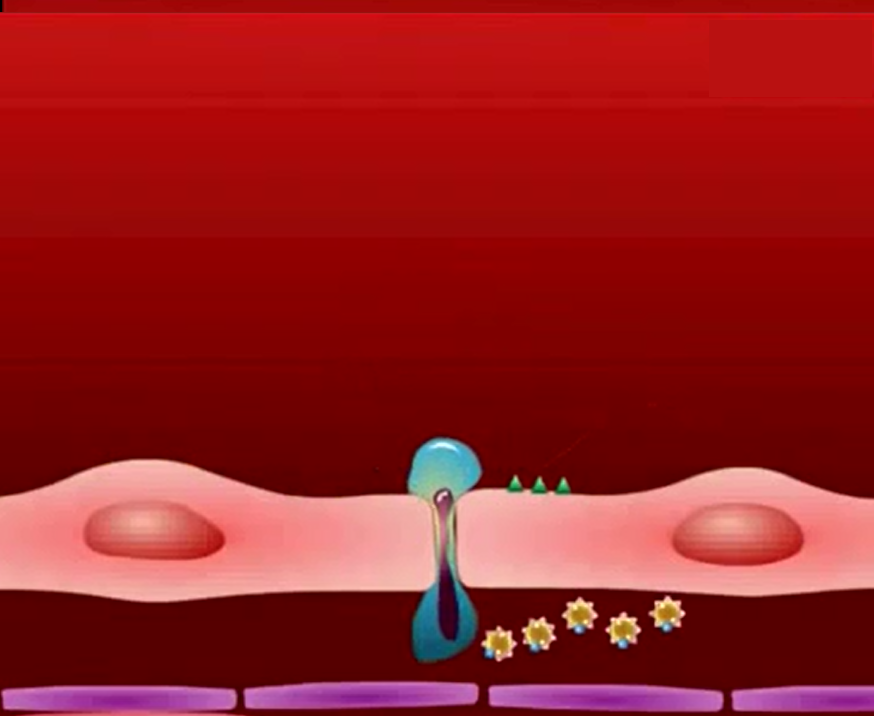}}
\quad
\subfloat[{\tiny  ingestion of L$_{ox}$ by M}]{\includegraphics[width=3.75cm]{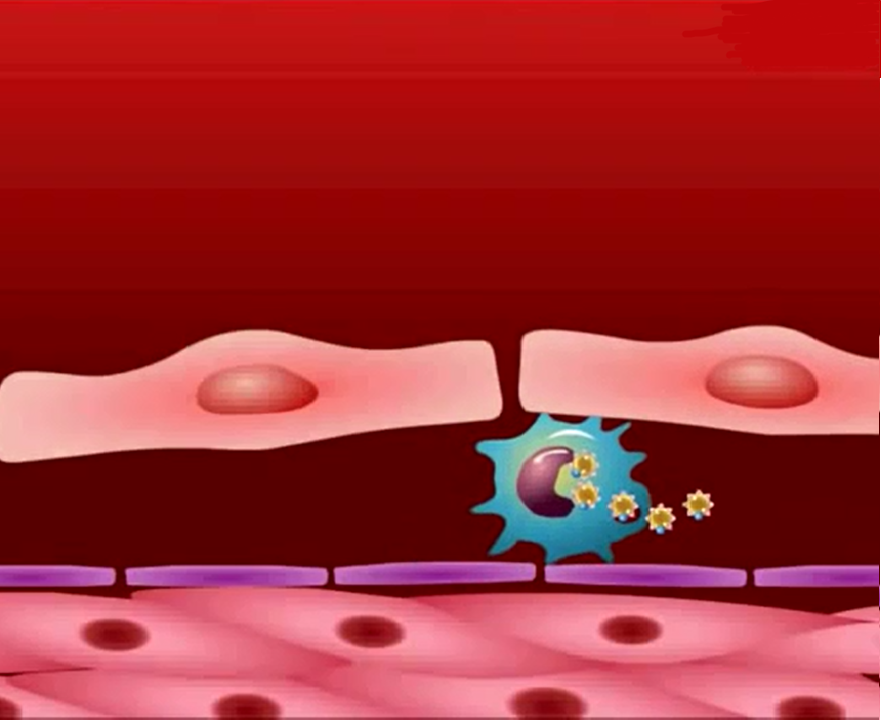}}
\quad
\subfloat[{\tiny F formation} ]{\includegraphics[width=3.75cm]{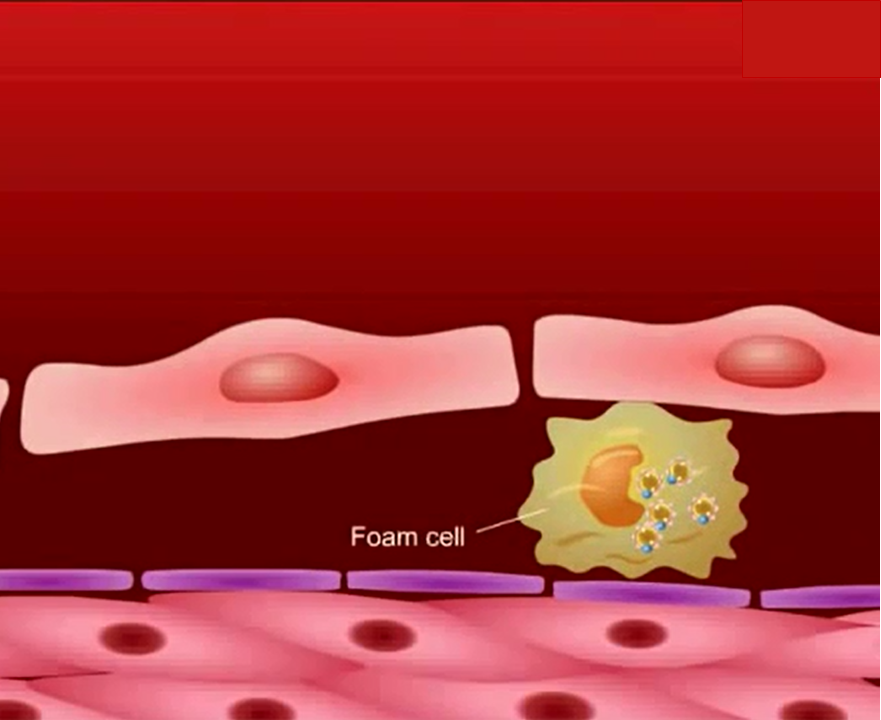}}
\quad
\subfloat[{\tiny Plaque formation} ]{\includegraphics[width=3.75cm]{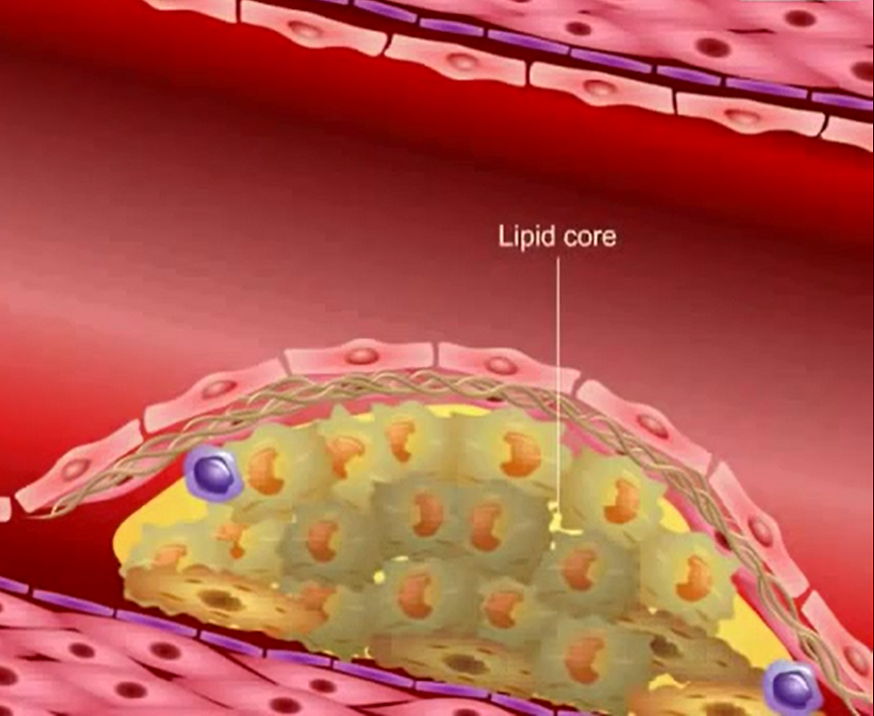}}
\caption{\it{The process of plaque development \cite{youttt}.}}
\label{f1}
\end{figure}
A plaque contains low density lipoprotein (LDL) or bad cholesterol,
 high-density lipoprotein (HDL) or good cholesterol, macrophages and foam cells (Figure \ref{f1}.a).
 The process of plaque development begins with a lesion in the
endothelial layer, penetration of low density lipoproteins in the intima (Figure \ref{f1}.b) and by Free radicals in the initma (Figure \ref{f1}.c) becoming oxidized LDL  (Figure \ref{f1}.d) and it is presented in the second term of (\ref{Lb}). Notice that  in this model $\widehat{L}$ (and $\widehat{H}$) and its oxidized form $\widehat{L}_{ox}$ (and $\widehat{H}_{ox}$) are merged.
Free radicals are oxidative agents continuously released by biochemical reactions within the body, including the intima \cite{singh2015role,bankson1993role}.
Endothelial cells, sensing the presence of ox-LDL, are activated and trigger 
monocyte chemoattractant protein, which triggers
recruitment of monocytes into the intima (Figure \ref{f1}.e) \cite{bankson1993role,gerrity1981role}. After entering the
intima, monocytes differentiate into macrophages  (M) (Figure \ref{f1}.f). The ingestion of large amounts of
ox-LDL (Figure \ref{f1}.g), that shown in the first term of (\ref{Lb}), transforms the fatty macrophages into the foam cells (Figure \ref{f1}.h) \cite{bentzon2014mechanisms}.
Newly formed foam cells secrete chemokines which attract more macrophages, 
and the plaque is gradually calcified (Figure \ref{f1}.i).
 At the same time that LDL enters the intima, high-density lipoprotein also enters the intima and becomes
oxidized by free radicals (FR) as presented by the second term of (\ref{Hb}). However, oxidized HDL ($H_{ox}$)
is not ingested by macrophages. HDL helps to prevent Atherosclerosis by removing cholesterol from foam cells, and by the limiting
inflammatory processes that underline Atherosclerosis, as shown in the first term of (\ref{Hb}). Furthermore, HDL takes up free radicals that are otherwise available
to LDL {\cite{hao2014ldl}}.
As the plaque continues to grow (Figure \ref{f1}.i), the increased shear force may cause rupture of
the plaque, possibly resulting in the formation of a thrombus (blood
clot)  and heart attack (for more information about the model see \cite{friedman2015free}).\\
For simplicity, we consider the model from cylindrical coordinates (\ref{Lb})-(\ref{etf}) to spherical coordinates as follows
\begin{eqnarray}\nonumber
&\dfrac{\partial \widehat{L}}{\partial t}-\dfrac{1}{r^2}\dfrac{\partial}{\partial r}(r^2\dfrac{\partial\widehat{L}}{\partial r})=-k_1 \dfrac{(M_0-\widehat{F})\widehat{L}}{K_1+\widehat{L}}-r_1\widehat{L},\,\, R(t)<r <1, \, t >0,&\\
&\dfrac{\partial \widehat{L}}{\partial r}+\alpha (\widehat{L}-L_0)=0\,\,\text{at}\,\, r =R(t), \, t>0,\quad \dfrac{\partial \widehat{L}}{\partial r}=0\,\,\text{at}\,\, r =1, \, t >0,\quad \widehat{L}(r,0)=L_0,&\label{e1}
\end{eqnarray}
\begin{eqnarray} \nonumber
&\dfrac{\partial \widehat{H}}{\partial t}-\dfrac{1}{r^2}\dfrac{\partial}{\partial r}(r^2\dfrac{\partial \widehat{H}}{\partial r})=-k_2 \dfrac{\widehat{H}\widehat{F}}{K_2+\widehat{F}}-r_2\widehat{H},\,\, R(t)<r <1, \, t >0,&\\
&\dfrac{\partial \widehat{H}}{\partial r}+\alpha (\widehat{H}-H_0)=0 \,\,\text{at}\,\, r=R(t), \, t >0,\quad \dfrac{\partial \widehat{H}}{\partial r}=0 \,\,\text{at}\,\, r =1, \, t >0,\quad \widehat{H}(r,0)=H_0,&\label{e2}
\end{eqnarray}
\begin{eqnarray} \nonumber
&\quad\dfrac{\partial \widehat{F}}{\partial t}-\dfrac{D}{r^2}\dfrac{\partial}{\partial r}(r^2\dfrac{\partial\widehat{F}}{\partial r})+\widehat{F}_r.v=k_1 \dfrac{(M_0-\widehat{F})\widehat{L}}{K1+\widehat{L}}-k_2\dfrac{\widehat{H}\widehat{F}}{K_2+\widehat{F}}-&\\ \nonumber
&\lambda \dfrac {\widehat{F}(M_0-\widehat{F})\widehat{L}}{M_0(\delta +\widehat{H})}+\dfrac{\mu_1}{M_0}(M_0-
\widehat{F})\widehat{F}-\dfrac{\mu_2}{M_0}\widehat{F}(M_0-\widehat{F}),\,\, R(t)<r <1, \, t >0,&\\
&\dfrac{\partial \widehat{F}}{\partial r}+\beta \widehat{F}=0 \,\,\text{at}\,\, R(t)<r <1, \, t >0,\quad \dfrac{\partial \widehat{F}}{\partial r}=0 \,\,\text{at}\,\, r =1, \, t >0,\quad \widehat{F}(r,0)=0,&
\end{eqnarray}
\begin{eqnarray}  \nonumber
&M_0.v_{r}=\lambda\dfrac{ (M_0-\widehat{F})\widehat{L}}{\delta +\widehat{H}}-\mu_1 (M_0-\widehat{F})-\mu_2 \widehat{F},\,\, R(t)<r <1, \, t >0,&\\
&v(r,t)=0 \,\,\text{at}\,\, r =1, \, t >0,&
 \end{eqnarray}
\begin{eqnarray} \nonumber
&\dfrac{d R(t)}{d t}=v(R(t),t), \, t>0,&\\
&R (0)=\epsilon .& \label{e3}
\end{eqnarray}
In the next section, we want to present a new reformulation of the model presented in (\ref{e1})-(\ref{e3}).
\section{A new reformulation of the model}\label{reform}
Most of the mathematical models that originate from biological phenomena have some special features. Free or moving boundary condition, mix (robin) boundary condition and many other features of a mathematical model can cause some difficulties in applying classical numerical methods on them and researchers overcome these difficulties by applying suitable techniques. Since the model of Atherosclerosis (\ref{e1})-{\ref{e3}) is a free boundary model with mix boundary condition, to solve this model numerically, we need to remove the mentioned difficulties. To reach this aim, we must consider a new reformulation of this model for which the free boundary change to a fixe one and the robin (mix) boundary condition change to a Neumann one.
\subsection{{Fixing the domain}}
Due to the fact that the plaque grows radially symmetric with free boundary, using the front fixing method by the following variable changes
\[\rho=\dfrac{r-R(t)}{1-R(t)},\quad\widehat{L}(\rho ,t)=\widehat{L}(r,t),\,\, \widehat{H}(\rho , t)=\widehat{H}(r,t), \,\, \widehat{F}(\rho , t)=\widehat{F}(r,t),\]
the free boundary problem (\ref{e1})-(\ref{e3})  is transformed into a problem with a fixed domain 
\[\{(\rho , t) \mid 0< \rho < 1 , t \geq 0\},\]
and the model becomes as follows
 \begin{eqnarray}\nonumber
&\dfrac{\partial \widehat{L}}{\partial t}-\dfrac{1}{(1-R(t))^2}\dfrac{\partial^2 \widehat{L}}{\partial \rho^2}+\left( \dfrac{-2}{\rho (1-R(t))^2+R(t)(1-R(t))}+ \dfrac{v(0,t)(\rho -1)}{(1-R(t))^2}\right) \dfrac{\partial \widehat{L}}{\partial \rho}=f^{\widehat{L}}(\widehat{L},\widehat{F}),&\\
&\dfrac{\partial \widehat{L}}{\partial \rho}+\alpha (1-R(t))(\widehat{L}-L_0)=0\,\,\text{at}\,\, \rho =0, \, t >0,\quad \dfrac{\partial \widehat{L}}{\partial \rho}=0\,\,\text{at}\,\, \rho =1, \, t >0,\quad \widehat{L}(\rho,0)=L_0,&\label{eqf1}
\end{eqnarray}
\begin{eqnarray}\nonumber
& \dfrac{\partial \widehat{H}}{\partial t}-\dfrac{1}{(1-R(t))^2}\dfrac{\partial^2 \widehat{H}}{\partial \rho^2}+\left( \dfrac{-2}{\rho (1-R(t))^2+R(t)(1-R(t))}+ \dfrac{v(0,t)(\rho -1)}{(1-R(t))^2}\right) \dfrac{\partial \widehat{H}}{\partial \rho}=f^{\widehat{H}}(\widehat{H},\widehat{F}),&\\
&\dfrac{\partial \widehat{H}}{\partial \rho}+\alpha(1-R(t))(\widehat{H}-H_0)=0 \,\,\text{at}\,\, \rho =0, \, t >0,\quad \dfrac{\partial\widehat{H}}{\partial \rho}=0 \,\,\text{at}\,\, \rho =1, \, t >0,\quad \widehat{H}(\rho,0)=H_0, &\label{eqfh}
\end{eqnarray}
\begin{eqnarray}\nonumber
 &\dfrac{\partial \widehat{F}}{\partial t}-\dfrac{D}{(1-R(t))^2}\dfrac{\partial^2 \widehat{F}}{\partial \rho^2}+&\\\nonumber
 &\left( \dfrac{-2D}{\rho (1-R(t))^2+R(t)(1-R(t))}+\dfrac{v}{1-R(t)}+ \dfrac{v(0,t)(\rho -1)}{(1-R(t))^2}\right) \dfrac{\partial \widehat{F}}{\partial \rho}=f^{\widehat{F}}(\widehat{L},\widehat{H},\widehat{F}),&\\
&\dfrac{\partial \widehat{F}}{\partial\rho}+\beta (1-R(t)) \widehat{F}=0 \,\,\text{at}\,\, \rho =0, \, t >0,\quad  \dfrac{\partial \widehat{F}}{\partial \rho}=0 \,\,\text{at}\,\, \rho =1, \, t >0,\quad \widehat{F}(\rho,0)=0,&\label{eqf2}
\end{eqnarray}
\begin{eqnarray}\nonumber
 &\dfrac{1}{1-R(t)}\dfrac{\partial v}{\partial \rho}=f^v(\widehat{L},\widehat{H},\widehat{F}),\\
&v(\rho,t)=0 \,\,\text{at}\,\, \rho =1, \, t >0,\label{eqfr}
\end{eqnarray}
\begin{eqnarray}\nonumber
& \dfrac{d R}{ d t}=v(0,t), \, t>0,\\ \label{RR}
&R (0)=\epsilon. \label{eqfe}
 \end{eqnarray}
 \subsection{changing the boundary condition}
To obtain more comfortable results for numerical analysis, without loss of generality, we can change the mixed boundary condition of the model to a Neumann one by applying suitable variable changes as follows
 \begin{eqnarray}
&&L(\rho,t):=\exp (-\alpha (1-R(t))\dfrac{(1-\rho)^2}{2})(\widehat{L}(\rho,t)-L_0),\\ \label{vc1}
&&H(\rho,t):=\exp (-\alpha (1-R(t))\dfrac{(1-\rho)^2}{2})(\widehat{H}(\rho,t)-H_0),\\\label{vc2}
&&F(\rho,t):=\exp (-\beta (1-R(t))\dfrac{(1-\rho)^2}{2})\widehat{F}(\rho,t).\label{vc3}
\end{eqnarray}
So, from (\ref{eqf1})-(\ref{RR}) we have
 \begin{eqnarray}\nonumber
&\dfrac{\partial {L}}{\partial t}-\dfrac{1}{(1-R(t))^2}\dfrac{\partial^2 {L}}{\partial \rho^2}+&\\\nonumber
&\left( \dfrac{-2}{\rho (1-R(t))^2+R(t)(1-R(t))}+ \dfrac{v(0,t)(\rho -1)}{1-R(t)}+\dfrac{2(1-\rho)\alpha}{1-R(t)}\right) \dfrac{\partial {L}}{\partial \rho}=f^L({L},{F}),&\\
&\dfrac{\partial {L}}{\partial \rho}=0\,\,\text{at}\,\, \rho =0, \, t >0,\quad \dfrac{\partial {L}}{\partial \rho}=0\,\,\text{at}\,\, \rho =1, \, t >0,\quad {L}(\rho,0)=0,&\label{fin1}
\end{eqnarray}
\begin{eqnarray}\nonumber
& \dfrac{\partial {H}}{\partial t}-\dfrac{1}{(1-R(t))^2}\dfrac{\partial^2 {H}}{\partial \rho^2}+&\\\nonumber
&\left( \dfrac{-2}{\rho (1-R(t))^2+R(t)(1-R(t))}+ \dfrac{v(0,t)(\rho -1)}{1-R(t)}+\dfrac{2(1-\rho)\alpha}{1-R(t)}\right) \dfrac{\partial {H}}{\partial \rho}=f^H({H},{F}),&\\
&\dfrac{\partial{H}}{\partial \rho}=0 \,\,\text{at}\,\, \rho =0, \, t >0,\quad \dfrac{\partial{H}}{\partial \rho}=0 \,\,\text{at}\,\, \rho =1, \, t >0,\quad {H}(\rho,0)=0, &
\end{eqnarray}
\begin{eqnarray}\nonumber
 &\dfrac{\partial{F}}{\partial t}-\dfrac{D}{(1-R(t))^2}\dfrac{\partial^2 {F}}{\partial \rho^2}+&\\\nonumber
 &\left( \dfrac{-2D}{\rho (1-R(t))^2+R(t)(1-R(t))}+\dfrac{v}{1-R(t)}+ \dfrac{v(0,t)(\rho -1)}{1-R(t)}+\dfrac{2D(1-\rho)\beta}{1-R(t)}\right) \dfrac{\partial{F}}{\partial \rho}=f^F({L},{H},{F})&\\
&\dfrac{\partial {F}}{\partial \rho}=0 \,\,\text{at}\,\, \rho =0, \, t >0,\quad\dfrac{\partial {F}}{\partial \rho}=0 \,\,\text{at}\,\, \rho =1, \, t >0,\quad{F}(\rho,0)=0,&\label{Fba}
\end{eqnarray}
\begin{eqnarray}\nonumber
 &\dfrac{1}{1-R(t)}\dfrac{\partial v}{\partial \rho}=f^v({L},{H},{F}),&\\
&v(\rho,t)=0 \,\,\text{at}\,\, \rho =1, \, t >0,&
\end{eqnarray}
\begin{eqnarray}\nonumber
& \dfrac{d R}{ d t}=v(0,t), \, t>0,&\\ 
&R (0)=\epsilon . &\label{fin2}
 \end{eqnarray}
\section{Approximating the solution of the problem}\label{sec3}
In this section, we approximate the solution of the problem (\ref{fin1})-(\ref{fin2})  for $0<\rho<1$ and $0<t<T$.  Let  $t_i :=ih\,\, (i=0,1,\ldots, M)$ be mesh points, where $h:=\dfrac{T}{M}$ is the  time step and $M$ is a positive integer.  The problem is solved using the finite difference-collocation method. In order to prove the stability and convergence of the method, we need to construct a non-classical discretization of second-order to approximate the time derivative.
 So, let us consider  the following discretization formula for approximating the time derivative for a given function $u(\rho,t)$
\begin{equation}\label{timed}
\dfrac{\partial u}{\partial t}(\rho,t_{n+1})=\dfrac{u_{n+1}-u_n+\dfrac{u_{n-1}-u_n}{3}}{\dfrac{2h}{3}} +E_t,
\end{equation}
and the following approximation formula for linearizing the equations
\begin{equation}
u(\rho ,t_{n+1})=2u(\rho,t_n)-u(\rho,t_{n-1})+E^u,
\end{equation}
where
\begin{equation}\label{err}
\max\{\Vert E_t\Vert_{\infty},\Vert E^u\Vert_{\infty}\} <C{h}^2,
\end{equation}
and $C$ is a positive constant. In the following, we have assumed that 
$$u_n(\rho)=u(\rho,t_n).$$
Using finite difference method based on the approximation formula given by (\ref{timed}) and equations (\ref{eqf1})-(\ref{eqf2}) we get 
 \begin{eqnarray}\nonumber
& L_{n+1}-L_n+\dfrac{L_{n-1}-L_n}{3}-\dfrac{2h}{3(1-R_{n+1})^2}\dfrac{\partial^2 L_{n+1}}{\partial \rho^2}+&\\ \nonumber
&\dfrac{2h}{3}\left( \dfrac{-2}{\rho (1-R_{n+1})^2+R_{n+1}(1-R_{n+1})}+ \dfrac{(2v_{n}(0)-v_{n-1}(0))(\rho -1)}{1-R_{n+1}}+\dfrac{2(1-\rho)\alpha}{1-R_{n+1}}\right) \dfrac{\partial L_{n+1}}{\partial \rho}=&\\ \nonumber
&\dfrac{2h}{3}(2f^L(L_n,F_n)-f^L(L_{n-1},F_{n-1}))-{\dfrac{2h}{3}}E_t^L,&\\ \nonumber
&\dfrac{\partial L_{n+1}}{\partial \rho}=0\,\,\text{at}\,\, \rho =0, \, t >0,\quad \dfrac{\partial L_{n+1}}{\partial \rho}=0\,\,\text{at}\,\, \rho =1, \, t >0,\quad L_0(\rho)=0,&
\end{eqnarray}
\begin{eqnarray}\nonumber
& H_{n+1}-H_n+\dfrac{H_{n-1}-H_{n}}{3}-\dfrac{2h}{3(1-R_{n+1})^2}\dfrac{\partial^2 H_{n+1}}{\partial \rho^2}+&\\\nonumber
&\dfrac{2h}{3}\left( \dfrac{-2}{\rho (1-R_{n+1})^2+R_{n+1}(1-R_{n+1})}+ \dfrac{(2v_{n}(0)-v_{n-1}(0))(\rho -1)}{1-R_{n+1}}+\dfrac{2(1-\rho)\alpha}{1-R_{n+1}}\right) \dfrac{\partial H_{n+1}}{\partial \rho}=&\\\nonumber
&\dfrac{2h}{3}(2f^H(H_n,F_n)-f^H(H_{n-1},F_{n-1}))-{\dfrac{2h}{3}}E_t^H,&\\\nonumber
&\dfrac{\partial H_{n+1}}{\partial \rho}=0 \,\,\text{at}\,\, \rho =0, \, t >0,\quad \dfrac{\partial H_{n+1}}{\partial \rho}=0 \,\,\text{at}\,\, \rho =1, \, t >0,\quad H_0(\rho)=0, &
\end{eqnarray}
\begin{eqnarray}\nonumber
 & F_{n+1}-F_n+\dfrac{F_{n-1}-F_n}{3}-\dfrac{2hD}{3(1-R_{n+1})^2}\dfrac{\partial^2 F_{n+1}}{\partial \rho^2}+&\\\nonumber
 &\dfrac{2h}{3}\left( \dfrac{-2D}{\rho (1-R_{n+1})^2+R_{n+1}(1-R_{n+1})}+\dfrac{(2v_{n}-v_{n-1})}{1-R_{n+1}}\right. +&\\\nonumber
 &\left. \dfrac{(2v_{n}(0)-v_{n-1}(0))(\rho -1)}{1-R_{n+1}}+\dfrac{2D(1-\rho)\beta}{1-R_{n+1}}\right) \dfrac{\partial F_{n+1}}{\partial \rho}=&\\\nonumber
 &\dfrac{2h}{3}(2f^F(L_n,H_n,F_n)-f^F(L_{n-1},H_{n-1},F_{n-1}))-{\dfrac{2h}{3}}E^F_t,&\\
&\dfrac{\partial F_{n+1}}{\partial \rho}=0 \,\,\text{at}\,\, \rho =0, \, t >0,\quad \dfrac{\partial F_{n+1}}{\partial \rho}=0 \,\,\text{at}\,\, \rho =1, \, t >0,\quad F_0(\rho)=0,&
\end{eqnarray}
\begin{eqnarray}\nonumber
 &\dfrac{1}{1-R_{n}}\dfrac{\partial v_n}{\partial \rho}=f^v(L_n,H_n,F_n),&\\
&v(\rho,t)=0 \,\,\text{at}\,\, \rho =1, \, t >0,&
 \end{eqnarray}
where $E_t^u$ is obtained by merging the errors of $E_t$ and $E^u$.
 From (\ref{err}), there exists positive $C^*$ such that
 \begin{equation}
\max\{\Vert E^L_t\Vert_{\infty} ,\Vert E^H_t\Vert_{\infty},\Vert E^F_t\Vert_{\infty}\}<C^*{h}^2,
\end{equation}
and using (\ref{RR}) we have
\begin{eqnarray}
R_{n+1}=R_{n}-\dfrac{R_{n-1}-R_n}{h}+2hv_n(0)-{\dfrac{2h}{3}}E^R_t.
\end{eqnarray}
Then, we conclude that there exists positive constant $C^{**}$ such that 
 \begin{equation}
max\{\Vert E^L_t\Vert_{\infty} ,\Vert E^H_t\Vert_{\infty},\Vert E^F_t\Vert_{\infty},\Vert E^R_t\Vert_{\infty}\}<C^{**} {h}^2.
\end{equation}
In the rest of this paper, the solution of the problem  (\ref{e1})-(\ref{e2}) is denoted by  $(L_{n+1}^{ap},H_{n+1}^{ap},F_{n+1}^{ap})$, which is the approximated solution of the following problem
 \begin{eqnarray}\nonumber
& L_{n+1}-\dfrac{h^*}{(1-R_{n+1}^{ap})^2}\dfrac{\partial^2 L_{n+1}}{\partial \rho^2}+&\\ \nonumber 
&{h^*}\left( \dfrac{-2}{\rho (1-R_{n+1}^{ap})^2+R_{n+1}^{ap}(1-R_{n+1}^{ap})}+ \dfrac{(2v_{n}^{ap}(0)-v_{n-1}^{ap}(0))(\rho -1)}{1-R_{n+1}^{ap}}+\dfrac{2(1-\rho)\alpha}{1-R_{n+1}^{ap}}\right) \dfrac{\partial L_{n+1}}{\partial \rho}=&\\ \nonumber
&h^*(2f^L(L_n^{ap},F_n^{ap})-f^L(L_{n-1}^{ap},F_{n-1}^{ap}))+L_n^{ap}-\dfrac{L^{ap}_{n-1}-L_n^{ap}}{3},&\\
&\dfrac{\partial L_{n+1}}{\partial \rho}=0\,\,\text{at}\,\, \rho =0, \, t >0,\quad\dfrac{\partial L_{n+1}}{\partial \rho}=0\,\,\text{at}\,\, \rho =1, \, t >0,\quad L_0(\rho)=0,&\label{apL}
\end{eqnarray}
\begin{eqnarray}\nonumber
& H_{n+1}-\dfrac{h^*}{(1-R_{n+1}^{ap})^2}\dfrac{\partial^2 H_{n+1}}{\partial \rho^2}+&\\\nonumber
&h^*\left( \dfrac{-2}{\rho (1-R_{n+1}^{ap})^2+R_{n+1}^{ap}(1-R_{n+1}^{ap})}+ \dfrac{(2v_{n}^{ap}(0)-v_{n-1}^{ap}(0))(\rho -1)}{1-R_{n+1}}+\dfrac{2(1-\rho)\alpha}{1-R_{n+1}^{ap}}\right) \dfrac{\partial H_{n+1}}{\partial \rho}=&\\\nonumber
&h^*(2f^H(H_n^{ap},F_n^{ap})-f^H(H_{n-1}^{ap},F_{n-1}^{ap}))+H_n^{ap}-\dfrac{H^{ap}_{n-1}-H_n^{ap}}{3},&\\
&\dfrac{\partial H_{n+1}}{\partial \rho}=0 \,\,\text{at}\,\, \rho =0, \, t >0,\quad \dfrac{\partial H_{n+1}}{\partial \rho}=0 \,\,\text{at}\,\, \rho =1, \, t >0,\quad H_0(\rho)=0, &\label{apH}
\end{eqnarray}
\begin{eqnarray}\nonumber
 & F_{n+1}-\dfrac{h^*D}{(1-R_{n+1}^{ap})^2}\dfrac{\partial^2 F_{n+1}}{\partial \rho^2}+&\\\nonumber
 &h^*\left( \dfrac{-2D}{\rho (1-R_{n+1}^{ap})^2+R_{n+1}^{ap}(1-R_{n+1}^{ap})}+\dfrac{(2v_{n}^{ap}-v_{n-1}^{ap})}{1-R_{n+1}^{ap}}\right. &\\\nonumber
 &\left. +\dfrac{(2v_{n}^{ap}(0)-v_{n-1}^{ap}(0))(\rho -1)}{1-R_{n+1}^{ap}}+\dfrac{2D(1-\rho)\beta}{1-R_{n+1}^{ap}}\right) \dfrac{\partial F_{n+1}}{\partial \rho}=&\\\nonumber
 &h^*(2f^F(L_n^{ap},H_n^{ap},F_n^{ap})-f^F(L_{n-1}^{ap},H_{n-1}^{ap},F_{n-1}^{ap}))+F_n^{ap}-\dfrac{F_{n-1}^{ap}-F_n^{ap}}{3},&\\
&\dfrac{\partial F_{n+1}}{\partial \rho}=0 \,\,\text{at}\,\, \rho =0, \, t >0,\quad \dfrac{\partial F_{n+1}}{\partial \rho}=0 \,\,\text{at}\,\, \rho =1, \, t >0,\quad F_0(\rho)=0,&\label{apF}
\end{eqnarray}
\begin{eqnarray}\nonumber
 &\dfrac{1}{1-R_{n+1}^{ap}}\dfrac{\partial v_n^{ap}}{\partial \rho}=f^v(L_n^{ap},H_n^{ap},F_n^{ap}),&\\
&v(\rho,t)=0 \,\,\text{at}\,\, \rho =1, \, t >0,&\label{apv}
 \end{eqnarray}
 \begin{eqnarray}
R_{n+1}^{ap}=R_{n}^{ap}-\dfrac{R_{n-1}^{ap}-R_{n}^{ap}}{3}+h^*v_n^{ap}(0),\label{apr}
\end{eqnarray}
where $h^*=\dfrac{2h}{3}$ and $(L_{n+1}^{ap},H_{n+1}^{ap},F_{n+1}^{ap})$ is obtained as an approximation of $(L_{n+1},H_{n+1},F_{n+1})$ by solving (\ref{apL})-(\ref{apr}) employing the collocation method.
To implement this method, we employ $\{p_i(\rho)\}_{i=0}^{\infty}$ as trial functions as follows
\begin{equation}\label{formbasis}
\text{span}\left\lbrace p_0(\rho),p_1(\rho),...,p_k(\rho)\right\rbrace =\{u\in \text{span}\{1,\rho,\rho^2,...,\rho^{k+2}\};  \dfrac{\partial u}{\partial \rho}\vert_{\rho=0}=0,\dfrac{\partial u}{\partial \rho}\vert_{\rho=1}=0\}.
\end{equation}
Then, we approximate $F_{n+1}$ by $F_{n+1}^N$ defined as follows
\begin{equation}
F_{n+1}^N(\rho)=\sum_{i=0}^{N} a_i^{n+1} p_i(\rho).
\end{equation}
Now, we consider the following equation
\begin{eqnarray}\label{pi-in}
 & \Pi_N^{0,0}F_{n+1}^N-\dfrac{h^*D}{(1-R_{n+1}^{ap})^2}\Pi_N^{0,0}\dfrac{\partial^2 F_{n+1}^N}{\partial \rho^2}+h^*\Pi_N^{0,0}(G(\rho)\dfrac{\partial F_{n+1}^N}{\partial \rho})=I_N^{0,0}g_n^*,&
\end{eqnarray}
where 
\begin{equation}\label{gns}
g_n^*=\Pi_N^{0,0}\left(F_n^{ap}-\dfrac{F_{n-1}^{ap}-F_n^{ap}}{3}\right)+\underbrace{h^*(2f^F(L_n^{ap},H_n^{ap},F_n^{ap})-f^F(L_{n-1}^{ap},H_{n-1}^{ap},F_{n-1}^{ap}))}_{f_n^*},
\end{equation}
and
\[
G(\rho)= \dfrac{-2D}{\rho (1-R_{n+1}^{ap})^2+R_{n+1}^{ap}(1-R_{n+1}^{ap})}+
\]
\begin{equation}\label{grho}
\dfrac{(2v_{n}^{ap}-v_{n-1}^{ap})}{1-R_{n+1}^{ap}} +\dfrac{(2v_{n}^{ap}(0)-v_{n-1}^{ap}(0))(\rho -1)}{1-R_{n+1}^{ap}}+\dfrac{2D(1-\rho)\beta}{1-R_{n+1}^{ap}},
\end{equation}
where $\Pi_N^{0,0}$ and $I_N^{0,0}$ are the orthogonal projection and Jacobi-Gauss-Lobatto interpolation operator with respect to $\rho$ on $[0,1]$, respectively. 
The approximation of the solution of (\ref{apL}) and (\ref{apH}) can be obtained by $L_{n+1}^N$ and $H_{n+1}^N$ in a similar manner which is used for $F_{n+1}^N$.
\section{Convergence and stability}\label{sec4}
As we are aware, the first and most important step in numerical analysis is to prove the stability and convergence of the proposed numerical method. To do so, we need to introduce some mathematical preliminaries.
\begin{definition}\label{defs}
\cite{canuto2006spectral} Suppose $\Omega=(a,b)^d,~d\in\Bbb N$ and $L^2_{w^{\alpha,\beta}} (\Omega)$ is the space of square integrable functions in $\Omega$. Now, we can define the following inner product on $L^2_{w^{\alpha,\beta}} (\Omega)$
\begin{equation*}
(u,v)_{w^{\alpha,\beta},\Omega}=\int_{\Omega} w^{\alpha,\beta}(X)u(X)v(X) dX,~~ \forall u,v\in L^2_{w^{\alpha,\beta}(\Omega)},
\end{equation*}
\begin{equation*}
\Vert u\Vert_{w^{\alpha,\beta},\Omega}=\left( \int_{\Omega} w^{\alpha,\beta}(X)(u(X))^2 dX\right) ^{\frac{1}{2}}, ~~\forall u\in L^2_{w^{\alpha,\beta}(\Omega)}.
\end{equation*}
\end{definition}
\begin{definition}\label{defp}
\cite{canuto2006spectral} Suppose ${\Bbb P_N}^d$ is the space of all $d$ dimensional algebraic polynomials of degree at most $N$ in each  variable. $\Pi_{N,w^{\alpha,\beta}}:L^2_{w^{\alpha,\beta}}(\Omega)\longrightarrow {\Bbb P}_N^d$ is an orthogonal projection if and only if for any $u\in L^2_{w^{\alpha,\beta}(\Omega)}$, we have
\begin{equation*}
\int_{\Omega} (\Pi_{N,w^{\alpha,\beta}} u(X)-u(X))v(X)w^{\alpha,\beta}(X) dX =0, ~~\forall v\in{\Bbb P}_N^d.
\end{equation*}
\end{definition}
\begin{theorem}\label{tha}
\cite{shen2011spectral} Let $\alpha ,\beta>-1$. For any $u\in B^m_{\alpha , \beta} (I)$, 
\begin{equation*}
\Vert \Pi_N^{\alpha , \alpha} u-u\Vert_{l,w^{\alpha ,\alpha}} \leq N^{2l-m-1/2} \Vert \partial_x^m u\Vert_{\alpha+m,\alpha+m}.
\end{equation*}
\end{theorem}
\subsection{Convergence}
In this section, we want to prove the convergence of the above method under the title of Convergence Theorem. For this purpose, using the principle of mathematical induction, we need to show that for $k=0,1,\ldots, M$, (for an arbitrary $M$) there exist positive constants $L^*$, $H^*$, $F^*$ and $R^*$ such that 
\begin{equation}\label{este1}
\vert L_k^{ap}-L_k\vert <L^*,~~\vert H_k^{ap}-H_k\vert <H^*,~~\vert F_k^{ap}-F_k\vert <F^*,~~\vert R_k^{ap}-R_k\vert <R^*,
\end{equation} 
where $L_k$, $H_k$, $F_k$ and $R_k$ are the exact solutions of (\ref{apL})-(\ref{apr}) in $t=t_k$, respectively.
First, we assume that
\begin{equation}\label{este2}
\vert L_k^{ap}-L_k\vert <L^*,~~\vert H_k^{ap}-H_k\vert <H^*,~~\vert F_k^{ap}-F_k\vert <F^*,~~\vert R_k^{ap}-R_k\vert <R^*\quad 1\leq k \leq n<M.
\end{equation}
In the following, to prove the convergence theorem, we need to present the following lemma.
\begin{lemma}\label{lema1}
Let $F$ be the exact solution of (\ref{Fba}) on $[0,1]\times [0,T]$, $F_{n+1}^{ap}=F_{n+1}^{N}, ~\forall n\geq 1$ and $\dfrac{\partial^2 F}{\partial \rho^2}$ be  $C^1$-smooth function. Then, for each  $0< \rho< 1$, there exist positive constants $c_1$, $c_2$ and $c$ such that
\[
(\dfrac{1}{2}-c_1h^*) \|\dfrac{\partial(F_{n+1}^N-F_{n+1})}{\partial \rho}\|_{w^{0,0}}^2\leq\dfrac{1}{2} \|\dfrac{\partial(F_{n}^{ap}-F_{n})}{\partial \rho}\|_{w^{0,0}}^2+
\]
\[
\sum_{i=0}^n h^*c_2\left(\|L_i-L_{i}^{ap}\|_{w^{0,0}}^2+\|H_{i}-H_{i}^{ap}\|_{w^{0,0}}^2 +\|F_{i}-F_{i}^{ap}\|_{w^{0,0}}^2\right)+\|E_F^*\|_{\infty}^2+h^*K_1(N),
\]
where $K(N)$ is the error generated by the spectral method and
\begin{equation}\label{errff}
\|E_F^*\|_{\infty}\leq c(h^*)^2, ~~ \lim_{N\rightarrow\infty} K_1(N)=0.
\end{equation}
\end{lemma}
\begin{proof}
For every $N\in \Bbb N$, there exists a polynomial $F^N_1$ such that
\[I_N^{0,0}F_1^N=F_1^N,~~\dfrac{\partial F_1^N }{\partial \rho}(0,t)=0,~~\dfrac{\partial F_1^N }{\partial \rho}(1,t)=0, ~~F_1^N(\rho,0)=0,~~0\leq t\leq T,\]
and
\begin{equation}\label{ABA1}
\lim_{N\rightarrow\infty} \Big(\|\dfrac{\partial{(F_1^{N}-F)}}{\partial\rho}\|^2_{\omega ^{0,0}}+\|F_1^N-F\|^2_{\omega^{0,0}}+\|\dfrac{G(\rho)\partial(F-F_1^{N})}{\partial\rho}\|^2_{\omega^{0,0}}+\|\dfrac{\partial^2(F-F_1^{N})}{\partial\rho^2}\|^2_{\omega^{0,0}}\Big)=0.
\end{equation}
Now, by taking the inner product of both sides of (\ref{pi-in}), we conclude that
\[
(\Pi_N^{0,0} \mathtt{L}F_{n+1}^N-\Pi_N^{0,0}\mathtt{L}F^N_{n+1,1},\dfrac{\partial^2 (F_{n+1}^N-F_{n+1,1}^N)}{\partial \rho^2})_{w^{0,0}} =
\]
\[
(I^{0,0}_N \underbrace{(g_n^*-\mathtt{L}F_{n+1,1}^N)}_{g_n^1},\dfrac{\partial^2 (F_{n+1}^N-F_{n+1,1}^N)}{\partial \rho^2})_{w^{0,0}},
\]
where $F_{n+1,1}^N(\rho):=F_1^N(\rho,t_{n+1})$ and for each $u\in C^2[0,1]$, $\mathtt{L}$ is defined as follows
\begin{equation}\label{lu}
\mathtt{L}u=u-\dfrac{h^*D}{(1-R_{n+1})}\dfrac{\partial^2 u}{\partial \rho^2}+h^*G(\rho)\dfrac{\partial u}{\partial \rho},
\end{equation}
where $G(\rho)$ is defined in (\ref{grho}) and is a bounded continuous  function. Therefore we have
\[
\left(F_{n+1}^N-F_{n+1,1}^N,\dfrac{\partial^2 (F_{n+1}^N-F_{n+1,1}^N)}{\partial \rho^2}\right)_{w^{0,0}}-
\]
\[
\left(  \dfrac{h^*D}{(1-R_{n+1}^{ap})^2}\dfrac{\partial^2 \left(F_{n+1}^N-F_{n+1,1}^N\right)}{\partial \rho^2},\dfrac{\partial^2(F_{n+1}^N-F_{n+1,1}^N)}{\partial \rho^2} \right)_{w^{0,0}}+
\]
\begin{eqnarray}\nonumber
&\left(h^* G(\rho)\dfrac{\partial (F_{n+1}^N-F_{n+1,1}^N)}{\partial \rho} ,\dfrac{\partial^2(F_{n+1}^N-F_{n+1,1}^N)}{\partial \rho^2} \right) _{w^{0,0}}&\\
&=(I^{0,0}_N g_n^1,\dfrac{\partial^2 (F_{n+1}^N-F_{n+1,1}^N)}{\partial \rho^2})_{w^{0,0}}.&
\end{eqnarray}
Then, by using Cauchy-Schwarz inequality there exists a constant $c_5$ such that
\begin{eqnarray}\nonumber
&\|\dfrac{\partial(F_{n+1}^N-F_{n+1,1}^N)}{\partial \rho}\|^2_{w^{0,0}} +\dfrac{Dh^*}{(1-R_{n+1})^2}\|\dfrac{\partial^2 (F^{N}_{n+1}- F^{N}_{n+1,1})}{\partial \rho^2}\|_{w^{0,0}}^2 \leq&\\
&c_5h^*\|\dfrac{\partial (F^{N}_{n+1}- F^{N}_{n+1,1})}{\partial \rho}\|_{w^{0,0}}^2 +\vert(I^{0,0}_N g_n^1,\dfrac{\partial^2 (F_{n+1}^N-F_{n+1,1}^N)}{\partial \rho^2})_{w^{0,0}}\vert .&\label{gn1}
\end{eqnarray}
By combining (\ref{gns}) with (\ref{gn1}) we obtain that
\begin{eqnarray}\nonumber
&\Big| (I_N^{0,0}g_n^1,\dfrac{\partial (F^{N}_{n+1}- F^{N}_{n+1,1})}{\partial \rho})_{w^{0,0}}\Big|\leq \Big| (I_N^{0,0} (g_n^*-\mathtt{L}F_{n+1,1}^N),\dfrac{\partial (F^{N}_{n+1}- F^{N}_{n+1,1})}{\partial \rho})_{w^{0,0}}\Big|\leq &\\\nonumber
&\Big| (F_n^{ap}-F_{n,1}^N -\dfrac{F_{n-1}^{ap}-F_{n-1,1}^N-F_n^{ap}+F_{n,1}^N}{3},\dfrac{\partial^2 (F_{n+1}^N-F_{n+1,1}^N)}{\partial \rho^2})_{w^{0,0}}\Big|+ &\\
&\Big|(\underbrace{I_N^{0,0}(\mathtt{L}F_{n+1,1}^N-F_{n,1}^N+\dfrac{F_{n-1,1}^N-F_{n,1}^N}{3}) -I_N^{0,0}f_n^*}_{g_n^2},\dfrac{\partial^2 (F_{n+1}^N-F_{n+1,1}^N)}{\partial \rho^2})_{w^{0,0}}\Big|. &\label{as2}
\end{eqnarray}
Therefore, from (\ref{gn1}) and (\ref{as2}) and using Cauchy-Schwarz inequality and Young inequality, we get that
\[
\|\dfrac{\partial (F_{n+1}^N-F_{n+1,1}^N)}{\partial \rho}\|_{w^{0,0}}^2+\dfrac{Dh^*}{(1-R_{n+1})^2}\|\dfrac{\partial^2 (F_{n+1}^N-F_{n+1,1}^N)}{\partial \rho^2}\|_{w^{0,0}}^2\leq
\]
\[
(g_n^2,\dfrac{\partial^2 (F_{n+1}^N-F_{n+1,1}^N)}{\partial \rho^2})_{w^{0,0}}+c_5h^*\| \dfrac{\partial (F_{n+1}^N-F_{n+1,1}^N)}{\partial \rho}\|^2_{w^{0,0}}+\dfrac{1}{2}\|\dfrac{\partial(F_{n}^{ap}-F_{n,1}^N)}{\partial \rho}\|_{w^{0,0}}^2+
 \]
 \[
 \dfrac{1}{2}\|\dfrac{\partial (F_{n+1}^N-F_{n+1,1}^N)}{\partial \rho}\|_{w^{0,0}}^2+\dfrac{1}{3} \Big|(\dfrac{\partial(F_{n-1}^{ap}-F_{n-1,1}^N-F_n^{ap}+F_{n,1}^N)}{\partial \rho},\dfrac{\partial(F_{n+1}^N-F_{n+1,1}^N)}{\partial \rho})_{w^{0,0}}\Big| 
.\]
So there exists a positive constant  $c_4$ such that
\begin{eqnarray}\nonumber
&\dfrac{1}{2}\|\dfrac{\partial (F_{n+1}^N-F_{n+1,1}^N)}{\partial \rho}\|_{w^{0,0}}^2+\dfrac{Dh^*}{2(1-R_{n+1})^2}\|\dfrac{\partial^2 (F_{n+1}^N-F_{n+1,1}^N)}{\partial \rho^2}\|_{w^{0,0}}^2\leq &\\\nonumber
&+h^*c_4\| 2f^{F}_n-f^F_{n-1}+E^F-2f_n^{F,ap}+f_{n-1}^{F,ap}\|_{w^{0,0}}^2+\vert (\mathtt{L}_n F_{n,1}^N -\mathtt{L}_n F ,\dfrac{\partial^2 (F_{n+1}^N-F_{n+1,1}^N)}{\partial \rho^2})_{w^{0,0}}\vert +&\\\nonumber
 &c_5h^*\| \dfrac{\partial (F_{n+1}^N-F_{n+1,1}^N)}{\partial \rho}\|^2_{w^{0,0}}+\dfrac{1}{2}\|\dfrac{\partial(F_{n}^{ap}-F_{n,1}^N)}{\partial \rho}\|_{w^{0,0}}^2+&\\
& \dfrac{1}{3} \Big|(\dfrac{\partial(F_{n-1}^{ap}-F_{n-1,1}^N-F_n^{ap}+F_{n,1}^N)}{\partial \rho},\dfrac{\partial(F_{n+1}^N-F_{n+1,1}^N)}{\partial \rho})_{w^{0,0}}\Big|,&\label{jay2}
\end{eqnarray}
where 
\begin{equation}
\mathtt{L}_n u=\mathtt{L}u_{n+1}-u_{n}+\dfrac{u_{n-1}-u_{n}}{3},
\end{equation}
and $\mathtt{L}$ is defined in (\ref{lu}). On the other hand, from (\ref{pi-in}) we have
\begin{eqnarray}\nonumber
&\Big|(\dfrac{\partial(\,{F_{n}^{ap}-F_{n,1}^N-F_{n+1}^{ap}+F_{n+1,1}^N}\,)}{\partial \rho},\dfrac{\partial(F_{n+1}^N-F_{n+1,1}^N)}{\partial \rho})_{w^{0,0}}\Big|+\dfrac{h^*D}{2(1-R_{n+1})^2}\|\dfrac{\partial^2(F_{n+1}^N-F_{n+1,1}^N)}{\partial \rho^2}\|_{w^{0,0}}^2\leq &\\\nonumber
 &h^*c_4\|\dfrac{\partial(F_{n+1}^N-F_{n+1,1}^N)}{\partial \rho}\|_{w^{0,0}}^2 +h^*M_4^F\| 2f^{F}_n-f^F_{n-1}+E^F-2f_n^{F,ap}+f_{n-1}^{F,ap}\|_{w^{0,0}}^2+&\\\nonumber
& \dfrac{1}{3} \Big|(\dfrac{\partial(F_{n}^{ap}-F_{n,1}^N-F_{n+1}^{ap}+F_{n+1,1}^N)}{\partial \rho},\dfrac{\partial(F_{n+1}^N-F_{n+1,1}^N)}{\partial \rho})_{w^{0,0}}\Big|+&\\
&\vert (\mathtt{L}_n F_{n,1}^N -\mathtt{L}_n F ,\dfrac{\partial^2 (F_{n+1}^N-F_{n+1,1}^N)}{\partial \rho^2})_{w^{0,0}}\vert, &\label{1endL}
 \end{eqnarray}
where $$f_i^{j,ap}=f^j(L_i^{ap},H_i^{ap},F_i^{ap}),~~f_i^{j}=f^j(L_i,H_i,F_i),~~ j\in\{L,H,F\}.$$
Then, by applying the recurrence relation (\ref{1endL}) repeatedly in (\ref{jay2}) we can conclude that
\begin{eqnarray}\nonumber
&(\dfrac{1}{2}-c_1h^*) \|\dfrac{\partial(F_{n+1}^N-F_{n+1,1}^N)}{\partial \rho}\|_{w^{0,0}}^2\leq \dfrac{1}{2}\|\dfrac{\partial(F_{n}^{ap}-F_{n,1}^N)}{\partial \rho}\|_{w^{0,0}}^2+&\\\nonumber
&\sum_{i=0}^n\dfrac{1}{3^i} h^*c_2\left(\| L_{i,1}^N-L_{i}^{ap}\|_{w^{0,0}}^2+\|H_{i,1}^N-H_{i}^{ap}\|_{w^{0,0}}^2 +\|F_{i,1}^N-F_{i}^{ap}\|_{w^{0,0}}^2\right)+&\\
&h^*\|E_F^*\|_{\infty}^2+h^*K_1(N). &\label{akhari}
\end{eqnarray}
where $\|E_F^*\|_{\infty}$ and $K_1(N)$ are as in (\ref{errff}).
\begin{flushright}
$ \blacksquare$
\end{flushright}
\end{proof}
\begin{lemma}\label{lemas}
Let $F_{n+1}^{ap}=F_{n+1}^{N}, ~\forall n\geq 1$ and $\dfrac{\partial^2 F}{\partial \rho^2}$ be  $C^1$-smooth function. Then, for each  $0< \rho< 1$, there exist positive constants $M^*$ and $C^*$ such that the following inequality holds
\begin{eqnarray}
&(F_{n+1}^N-F_{n+1,1}^N)^2\Big|_{\rho=1} \leq F^* +&\\
&4(\dfrac{1}{3})^{n-1}\left( \dfrac{h^*D}{2(1-R_{n+1})^2}+\dfrac{C^*}{\sqrt{N}}\right) ({L^*}^2+{H^*}^2+{F^*}^2+h^*\Vert E_F^*\Vert_{\infty}^2+h^*K_1(N)) \leq M^*.&\nonumber
\end{eqnarray}
where $L^*, H^*$ and $F^*$ are obtained from (\ref{este1}) and $\|E_F^*\|_{\infty}$ and $K_1(N)$ are as in (\ref{errff}).
\end{lemma}
\begin{proof}
By taking the inner product of both sides of (\ref{pi-in}), we can deduce that
\[
(\Pi_N^{0,0} \mathtt{L}F_{n+1}^N-\Pi_N^{0,0}\mathtt{L}F^N_{n+1,1},\rho\dfrac{\partial}{\partial \rho}(\rho\dfrac{\partial (F_{n+1}^N-F_{n+1,1}^N)}{\partial \rho}))=\]
\[
(I^{0,0}_N (g_n^*-\mathtt{L}F_{n+1,1}^N), \rho\dfrac{\partial}{\partial \rho}(\rho\dfrac{\partial (F_{n+1}^N-F_{n+1,1}^N)}{\partial \rho})).
\]
Using (\ref{akhari}), (\ref{este1}), Theorem \ref{tha}, Definition \ref{defp}, Cauchy-Schwarz inequality and Young inequality, we have
\begin{eqnarray}\nonumber
&(F_{n+1}^N-F_{n+1,1}^N)^2\Big|_{\rho=1} \leq F^* +&\\
&4(\dfrac{1}{3})^{n-1}\left( \dfrac{h^*D}{2(1-R_{n+1})^2}+\dfrac{\max\{1,c_5\}}{\sqrt{N}}\right) ({L^*}^2+{H^*}^2+{F^*}^2+h^*\Vert E_F^*\Vert_{\infty}^2+h^*K_1(N)) \leq M^*,&
\end{eqnarray}
where $c_5$ is obtained from (\ref{gn1}). $\blacksquare$
\end{proof}
Finally, using Lemmas \ref{lema1} and \ref{lemas} we have
\begin{eqnarray}\nonumber
&(\dfrac{1}{2}-c_1h^*) \|\dfrac{\partial(F_{n+1}^N-F_{n+1,1}^N)}{\partial \rho}\|_{w^{0,0}}^2\leq \dfrac{1}{2}\|\dfrac{\partial(F_{n}^{ap}-F_{n,1}^N)}{\partial \rho}\|_{w^{0,0}}^2+&\\\nonumber
&\sum_{i=0}^n\dfrac{1}{3^i} h^*c_2\left(\|\dfrac{\partial (L_{i,1}^N-L_{i}^{ap})}{\partial \rho}\|_{w^{0,0}}^2+\|\dfrac{\partial (H_{i,1}^N-H_{i}^{ap})}{\partial \rho}\|_{w^{0,0}}^2 +\|\dfrac{\partial (F_{i,1}^N-F_{i}^{ap})}{\partial \rho}\|_{w^{0,0}}^2\right)+\\\label{lastF}
&h^*\|E_F^*\|_{\infty}^2+h^*K_1(N),&
\end{eqnarray}
where $\|E_F^*\|_{\infty}$ and $K_1(N)$ are as in (\ref{errff}).\\
Similar to the proof of lemma \ref{lema1}, we can show that there exist positive constants $c_6,\, c_7,\, c_8,\,c_9$ such that for $0\leq t\leq T,$ 
\begin{eqnarray}\nonumber
&(\dfrac{1}{2}-c_6h^*) \|\dfrac{\partial(L_{n+1}^N-L_{n+1,1}^N)}{\partial \rho}\|_{w^{0,0}}^2\leq \dfrac{1}{2}\|\dfrac{\partial(L_{n}^{ap}-L_{n,1}^N)}{\partial \rho}\|_{w^{0,0}}^2+&\\
 &\sum_{i=0}^n\dfrac{1}{3^i} h^*c_7\left( \|\dfrac{\partial(L_{i,1}^N-L_{i}^{ap})}{\partial \rho}\|_{w^{0,0}}^2 +\|\dfrac{\partial(F_{i,1}^N-F_{i}^{ap})}{\partial \rho}\|_{w^{0,0}}^2\right)+h^*\|E_L^*\|_{\infty}^2+h^*K_2(N),&\label{akhariH}
\end{eqnarray}
\begin{eqnarray}\nonumber
&(\dfrac{1}{2}-c_8h^*) \|\dfrac{\partial(H_{n+1}^N-H_{n+1,1}^N)}{\partial \rho}\|_{w^{0,0}}^2\leq \dfrac{1}{2}\|\dfrac{\partial(H_{n}^{ap}-H_{n,1}^N)}{\partial \rho}\|_{w^{0,0}}^2+&\\
& \sum_{i=0}^n\dfrac{1}{3^i} h^*c_9\left( \|\dfrac{\partial(H_{i,1}^N-H_{i}^{ap})}{\partial \rho}\|_{w^{0,0}}^2 +\|\dfrac{\partial(F_{i,1}^N-F_{i}^{ap})}{\partial \rho}\|_{w^{0,0}}^2\right)+h^*\|E_H^*\|_{\infty}^2+h^*K_3(N),&\label{akhariF}
\end{eqnarray}
where  
$$\lim_{N\rightarrow\infty} K_2(N)=0, \lim_{N\rightarrow\infty} K_3(N)=0~~\text{ and}~~  \max\{\Vert E_L^*\Vert_{\infty},\Vert E_H^*\Vert_{\infty}\}<c({h^*})^2, $$ 
and $c$ is a positive constant.\\
In the following theorem, the convergence of the proposed method is proved.
\begin{theorem}(Convergence Theorem)\label{convergetheo}
Let $L_{n+1}^{ap}=L_{n+1}^N$, $H_{n+1}^{ap}=H_{n+1}^N$ and $F_{n+1}^{ap}=F_{n+1}^N$. Under the assumption of Lemma \ref{lema1} there exist positive constants $M_1 ,M_ 2$ such that 
$$\max_{k=0,\ldots,n+1}\{\xi_{k}\}\leq M_1(e^{M_2T})({h^*}^{2}+(K(N))^{\frac{1}{2}}),$$
where
$$\xi_{k} = \Vert \dfrac {\partial(L_k^{ap}-L_k)}{\partial \rho}\Vert_{w^{0,0}}+\Vert \dfrac{\partial(H_k^{ap}-H_k)}{\partial \rho}\Vert_{w^{0,0}}+\Vert \dfrac{\partial ( F_k^{ap}-F_k)}{\partial \rho}\Vert_{w^{0,0}}+\vert R_k^{ap}-R_k\vert,$$ $$\lim_{N\rightarrow\infty}K(N)=0.$$
\end{theorem}
\begin{proof}
From (\ref{lastF}), (\ref{akhariH}) and (\ref{akhariF}) we can conclude that there exists a constant $M_4,$ such that 
\[
\max_{k=0,\ldots,n+1}\{\phi_{k}\}\leq (1+M_4h^*)\phi_{n} +(1+M_4h^*)^2\max{\{\phi_{k}\}}_{k=0,1,\ldots n}+(1+M_4h^*)(h^*\|E_F^*\|^2_{\infty}+h^*K(N)),
\]
where
$${\phi}_{k} = \Vert \dfrac {\partial(L_k^{ap}-L_k)}{\partial \rho}\Vert_{w^{0,0}}^2+\Vert \dfrac{\partial(H_k^{ap}-H_k)}{\partial \rho}\Vert_{w^{0,0}}^2+\Vert \dfrac{\partial ( F_k^{ap}-F_k)}{\partial \rho}\Vert_{w^{0,0}}^2+\vert R_k^{ap}-R_k\vert^2,$$ $$\lim_{N\rightarrow\infty}K(N)=0.$$
and $\|E_F^*\|$ is as in (\ref{errff}). Therefore, there exists a constant $M_3$ that
\[
\max_{k=0,\ldots,n+1}\{\phi_{k}\}\leq (1+M_3h^*)\max{\{\phi_{k}\}}_{k=0,1,\ldots n}+(1+M_3h^*)(h^*\|E_F^*\|^2_{\infty}+h^*K(N)).
\]
Then, by applying the above recurrence relation we have
\[
\max_{k=0,\ldots,n+1}\{\phi_{k}\}\leq (1+M_3h^*)^{n+1}\phi_0+\vert\dfrac{(1+M_3h^*)^{n+1}-1}{M_3h^*}\vert(h^*\|E_F^*\|^2_{\infty}+h^*K(N)).\]
Finally, it may be concluded that there exist constants $M_1$ and $M_2$ such that
\begin{equation}
\max_{k=0,\ldots,n+1}\{\xi_{k}\}\leq M_1(e^{{M_2T}})({h^*}^{2}+(K(N))^{\frac{1}{2}}).
\end{equation}\label{conlast}
\begin{flushright}
$\blacksquare$
\end{flushright}
\end{proof}
Employing the General Sobolev inequalities, there exists a positive constant $M_5$ such that for $0\leq t\leq T$,
\begin{equation}\label{idn}
\vert F_{n+1}^{ap}-F_{n+1}\vert \leq M_5 \Vert F_{n+1}^{ap}-F_{n+1}\Vert_{w^{0,0}}.
\end{equation}
Using (\ref{idn}) and (\ref{lastF}), we can choose proper $N$ and $h$ such that
 \begin{equation}\label{idn2}
\vert F_{n+1}^{ap}-F_{n+1}\vert\leq F^*.
\end{equation}
Similar to the proof of the Theorem {\ref{convergetheo}} and (\ref{idn2}), we can show that
 \begin{equation}\label{conlast}
\vert L_{n+1}^{ap}-L_{n+1}\vert\leq L^*,
\end{equation}
and 
 \begin{equation}\label{idn4}
\vert H_{n+1}^{ap}-H_{n+1}\vert\leq H^*.
\end{equation}
Thus, using (\ref{idn2}), (\ref{conlast}), (\ref{idn4}) and Theorem \ref{convergetheo}, we can deduce that the sequence $\{L_n,H_n,F_n,R_n\}_{n=0}^{\infty}$ converges to the exact solution of the problem (\ref{e1})-(\ref{e3}) on $[0,1]\times [0,T]$. Now, using the principle of mathematical induction, Theorem (\ref{convergetheo}), (\ref{este1}),  (\ref{este2}), (\ref{idn2}), (\ref{conlast}) and (\ref{idn4}), we have
\begin{equation}
\vert L_k^{ap}-L_k\vert <L^*,~~ \vert H_k^{ap}-H_k\vert <H^*,~~ \vert F_k^{ap}-F_k\vert <F^*,~~ \vert R_k^{ap}-R_k\vert <R^*, ~~0\leq k\leq M.
\end{equation}
\subsection{Stability}
In this section, we want to prove the stability of the presented method. For this purpose, first, we consider the perturbed problem as follows
 \begin{eqnarray}\nonumber
&\dfrac{\partial {L}}{\partial t}-\dfrac{1}{(1-R(t))^2}\dfrac{\partial^2 {L}}{\partial \rho^2}+&\\\nonumber
&\left( \dfrac{-2}{\rho (1-R(t))^2+R(t)(1-R(t))}+ \dfrac{v(0,t)(\rho -1)}{1-R(t)}+\dfrac{2(1-\rho)\alpha}{1-R(t)}\right) \dfrac{\partial {L}}{\partial \rho}=f^L({L},{F})+p_1(\rho,t),&\\
&\dfrac{\partial {L}}{\partial \rho}=0\,\,\text{at}\,\, \rho =0, \, t >0,\quad \dfrac{\partial {L}}{\partial \rho}=0\,\,\text{at}\,\, \rho =1, \, t >0,\quad {L}(\rho,0)=0,&\label{per1}
\end{eqnarray}
\begin{eqnarray}\nonumber
& \dfrac{\partial {H}}{\partial t}-\dfrac{1}{(1-R(t))^2}\dfrac{\partial^2 {H}}{\partial \rho^2}+&\\\nonumber
&\left( \dfrac{-2}{\rho (1-R(t))^2+R(t)(1-R(t))}+ \dfrac{v(0,t)(\rho -1)}{1-R(t)}+\dfrac{2(1-\rho)\alpha}{1-R(t)}\right) \dfrac{\partial {H}}{\partial \rho}=f^H({H},{F})+p_2(\rho,t),&\\
&\dfrac{\partial{H}}{\partial \rho}=0 \,\,\text{at}\,\, \rho =0, \, t >0,\quad \dfrac{\partial{H}}{\partial \rho}=0 \,\,\text{at}\,\, \rho =1, \, t >0,\quad {H}(\rho,0)=0, &
\end{eqnarray}
\begin{eqnarray}\nonumber
 &\dfrac{\partial{F}}{\partial t}-\dfrac{D}{(1-R(t))^2}\dfrac{\partial^2 {F}}{\partial \rho^2}+&\\\nonumber
 &\left( \dfrac{-2D}{\rho (1-R(t))^2+R(t)(1-R(t))}+\dfrac{v}{1-R(t)}+ \dfrac{v(0,t)(\rho -1)}{1-R(t)}+\dfrac{2D(1-\rho)\beta}{1-R(t)}\right) \dfrac{\partial{F}}{\partial \rho}&\\
 &=f^F({L},{H},{F})+p_3(\rho,t),&\\
&\dfrac{\partial {F}}{\partial \rho}=0 \,\,\text{at}\,\, \rho =0, \, t >0,\quad\dfrac{\partial {F}}{\partial \rho}=0 \,\,\text{at}\,\, \rho =1, \, t >0,\quad{F}(\rho,0)=0,&\label{per2}
\end{eqnarray}
\begin{eqnarray}\nonumber
 &\dfrac{1}{1-R(t)}\dfrac{\partial v}{\partial \rho}=f^v({L},{H},{F})+p_4(\rho,t),&\\
&v(\rho,t)=0 \,\,\text{at}\,\, \rho =1, \, t >0,&
\end{eqnarray}
\begin{eqnarray}\nonumber
& \dfrac{d R}{ d t}=v(0,t), \, t>0,&\\ \label{per4}
&R (0)=\epsilon. &
 \end{eqnarray}
 In the following theorem, the stability of the proposed method is proved.
\begin{theorem}\label{stabtheo}
 Let $\epsilon_1$ be a positive constant and $\vert p_i\vert<\epsilon_1 \,\,(i=1,\ldots , 4)$. Then, under the assumptions of Lemma \ref{lema1}, there exist positive constants $M_1^*,\, M_2^*$such that
$$\max_{k=0,\ldots ,n+1}{\{\xi_{k}\}}\leq M_1^*(e^{M_{2}^*T}-1)({h^*}^2+\epsilon_1+(K(N))^{\frac{1}{2}}),$$
where
$$\xi_{k} = \Vert \dfrac {\partial(L_k^{ap}-L_k)}{\partial \rho}\Vert_{w^{0,0}}+\Vert \dfrac{\partial(H_k^{ap}-H_k)}{\partial \rho}\Vert_{w^{0,0}}+\Vert \dfrac{\partial ( F_k^{ap}-F_k)}{\partial \rho}\Vert_{w^{0,0}}+\vert R_k^{ap}-R_k\vert.$$
\end{theorem}
\begin{proof}
If we solve the perturbed problem (\ref{per1})-(\ref{per4}) using the presented method, one can conclude that there exist a positive constant $M_4^*$ such that
\[
\max_{k=0,\ldots,n+1}\{\phi_{k}\}\leq (1+M_4^*h^*)\phi_{n} +(1+M_4^*h^*)^2\max{\{\phi_{k}\}}_{k=0,1,\ldots n}+(1+M_4^*h^*)(h^*\|E\|^2_{\infty}+h^*K(N)+h^*\epsilon_1),
\]
where $\|E\|_{\infty}\leq M^*{h^*}^2$ and $M^*$ is a positive constant and also,
$${\phi}_{k} = \Vert \dfrac {\partial(L_k^{ap}-L_k)}{\partial \rho}\Vert_{w^{0,0}}^2+\Vert \dfrac{\partial(H_k^{ap}-H_k)}{\partial \rho}\Vert_{w^{0,0}}^2+\Vert \dfrac{\partial ( F_k^{ap}-F_k)}{\partial \rho}\Vert_{w^{0,0}}^2+\vert R_k^{ap}-R_k\vert^2,$$ $$\lim_{N\rightarrow\infty}K(N)=0.$$
So we get
\[
\max_{k=0,\ldots,n+1}\{\phi_{k}\}\leq (1+M_3^*h^*)\max{\{\phi_{k}\}}_{k=0,1,\ldots n}+(1+M_3^*h^*)(h^*\|E\|^2_{\infty}+h^*K(N)+h^*\epsilon_1),
\]
then, by applying the above recurrence relation we have
\[
\max_{k=0,\ldots,n+1}\{\phi_{k}\}\leq (1+M_3^*h^*)^{n+1}\phi_0+\vert\dfrac{(1+M_3^*h^*)^{n+1}-1}{M_3^*h^*}\vert(h^*\|E\|^2_{\infty}+h^*K(N)+h^*\epsilon_1).\]
Therefore, there exist costats $M_1^*$ and $M_2^*$ such that we have
$$\max_{k=0,\ldots,n+1}\{\xi_{k}\}\leq M_1^*(e^{{M_2^*T}})({h^*}^{2}+\epsilon_1+(K(N))^{\frac{1}{2}}).$$
\begin{flushright}
$\blacksquare$
\end{flushright}
\end{proof}
\section{Numerical experiments}\label{sec5}
The main target of this section is to investigate the numerical solution of the model of Atherosclerosis from two perspectives: 1-numerical solution point of view 2-biological simulation point of view. We first examine the numerical results from the first perspective.
We solve the model of Atherosclerosis by applying the finite difference/collocation method. The most important question to ask is how to construct trial functions which satisfy the boundary condition to establish the collocation method. The first and simple way that comes to mind is to use a linear combination of monomial polynomials. So, we approximate the functions $L(\rho,t) ,H(\rho,t), F(\rho,t)$  in the form of (\ref{formbasis}) as follows
\[L_{n+1}^N(\rho)=\sum_{i=0}^N l_i^{n+1} p_i(\rho) ,~~H_{n+1}^N(\rho)=\sum_{i=0}^N h_i^{n+1} p_i(\rho) ,~~F_{n+1}^N(\rho)=\sum_{i=0}^N f_i^{n+1} p_i(\rho),\]
where 
\[p_i(\rho)=\dfrac{\rho^{i+2}}{i+2}-\dfrac{\rho^{i+1}}{i+1},~~ i=1,2,\ldots, N,\]
which will denote by "TFBM" (Trial Functions Based on Monomials). Also, the Gauss quadrature  points $\{x_i^{0,0}\}_{i=1}^{N}$ (i.e., the zeros of Legendre polynomial of degree $N+1$) are considered as  collocation points. The typical parameter values and the initial conditions
for our numerical simulations are:
$k_1=10,\,k_2=10,\,K_1=10^{-2},\, K_2=0.5,\, D=8.64\times 10^{-7}, \, \mu_1=0.015,\, \mu_2=0.03,\, r_1=2.42\times 10^{-5},\, r_2=5.45\times 10^{-7},\, \lambda=2.573\times 10^{-3},\, \delta=-2.541\times 10^{-3},\, M_0=5\times 10^{-5},\, \alpha=1,\, \beta=0.01,\,L_0=14\times 10^{-4}, \,H_0=6\times 10^{-4},\,F_0=0$, which are taken from \cite{hao2014ldl,friedman2015free}. Notice that in our simulations, the radius of initial plaque $\epsilon$  in (\ref{etf}) is considered to be 0.9 cm.

Now, in order to verify our numerical results, we need to present the following definition.
\begin{definition}
A sequence $\{x_n\}_{n=1}^{\infty}$ is said to converge to $x$ with order $p$ if there exists a constant $C$ such that $\vert x_n-x\vert \leq Cn^{-p} ,~~\forall n$.
This can be written as $\vert x_n-x\vert=\mathcal{O}(n^{-p})$.  A practical method to calculate the rate of convergence for a discretization method is to use the following formula
\begin{equation}\label{ratio_remark}
p\approx \dfrac{\log_e(e_{n_2}/e_{n_1})}{\log_e(n_1/n_2)},
\end{equation}
where $e_{n_1}$ and $e_{n_2}$ denote the errors with respect to the step sizes $\dfrac{1}{n_1}$ and $\dfrac{1}{n_2}$, respectively \cite{gautschi1997numerical}.
\end{definition}
It is valuable to point out that our numerical calculations are carried out using the \textbf{MATLAB 2018a} program in a computer with the Intel Core i7 processor (2.90 GHz, 4 physical cores). \\
Lies in the fact that it is hard or sometimes impossible to reach the exact solution of most of the coupled nonlinear models analytically, and because we have shown that the presented method for solving the model is stable and convergent (See Theorem \ref{convergetheo}), so, we consider numerical results for the large $M=300$ and $N=10$ as an exact solution and for other values of $M$, we report the time-error in Table \ref{timeerrTFBM}. To better see the time-error of the presented approach numerically, we report the obtained results in Figure \ref{fig_time2TFBM}. As we see in Table \ref{NTFBM} and Figure \ref{FTFBM}, the non-classical finite difference method presented in (\ref{timed}) has almost $\mathcal{O}(h^2)$  error, which verifies our theoretical results presented in Theorem \ref{convergetheo}.
\begin{table}
\begin{center}
{\scriptsize{\begin{tabular}{lllllllllllll}
\hline
{Error of} &  & \multicolumn{1}{c}{M=100} & \multicolumn{1}{c}{} & \multicolumn{1}{c}{M=200} & \multicolumn{1}{c}{} & \multicolumn{1}{c}{M=300} & \multicolumn{1}{c}{} & \multicolumn{1}{c}{M=400} & \multicolumn{1}{c}{} & \multicolumn{1}{c}{M=500} & \multicolumn{1}{c}{} & \multicolumn{1}{c}{M=600} \\ 
\cline{1-1}\cline{3-3}\cline{5-5}\cline{7-7}\cline{9-9}\cline{11-11}\cline{13-13}
\multicolumn{1}{c}{L} &  & \multicolumn{1}{c}{$1.832329e-05$} & \multicolumn{1}{c}{} & \multicolumn{1}{c}{$7.33126e-06$} & \multicolumn{1}{c}{} & \multicolumn{1}{c}{$4.11065e-06$} & \multicolumn{1}{c}{} & \multicolumn{1}{c}{$2.58829e-06$} & \multicolumn{1}{c}{} & \multicolumn{1}{c}{$1.70362e-06$} & \multicolumn{1}{c}{} & \multicolumn{1}{c}{$1.12596e-06$} \\ 
\multicolumn{1}{c}{H} &  & \multicolumn{1}{c}{$2.57909e-07$} & \multicolumn{1}{c}{} & \multicolumn{1}{c}{$1.03322e-07$} & \multicolumn{1}{c}{} & \multicolumn{1}{c}{$5.79623e-08$} & \multicolumn{1}{c}{} & \multicolumn{1}{c}{$3.65061e-08$} & \multicolumn{1}{c}{} & \multicolumn{1}{c}{$2.40324e-08$} & \multicolumn{1}{c}{} & \multicolumn{1}{c}{$1.58853e-08$} \\ 
\multicolumn{1}{c}{F} &  & \multicolumn{1}{c}{$6.60404e-09$} & \multicolumn{1}{c}{} & \multicolumn{1}{c}{$4.21233e-09$} & \multicolumn{1}{c}{} & \multicolumn{1}{c}{$3.16373e-09$} & \multicolumn{1}{c}{} & \multicolumn{1}{c}{$2.26606e-09$} & \multicolumn{1}{c}{} & \multicolumn{1}{c}{$1.60468e-09$} & \multicolumn{1}{c}{} & \multicolumn{1}{c}{$1.11186e-09$} \\ 
\hline
\end{tabular}}}
\end{center}
\caption{\it{Maximum time-error with N=10 and various M by considering TFBM.}}
\label{timeerrTFBM}
\end{table}
\begin{figure}
\centering
\includegraphics[scale=.7]{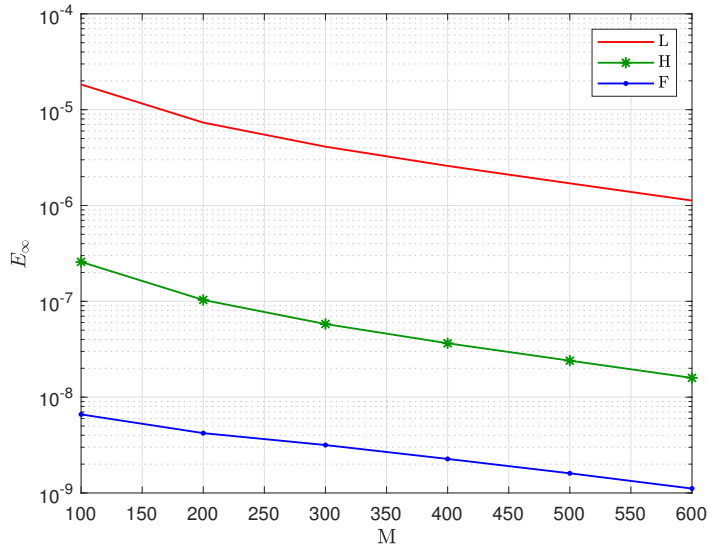}
\caption{\it{Maximum time-error functions with N=10 and various M by considering TFBM.}}
\label{fig_time2TFBM}
\end{figure}
\begin{minipage}{\textwidth}
	\begin{minipage}[b]{0.45\textwidth}
		\centering
	\includegraphics[scale=.7]{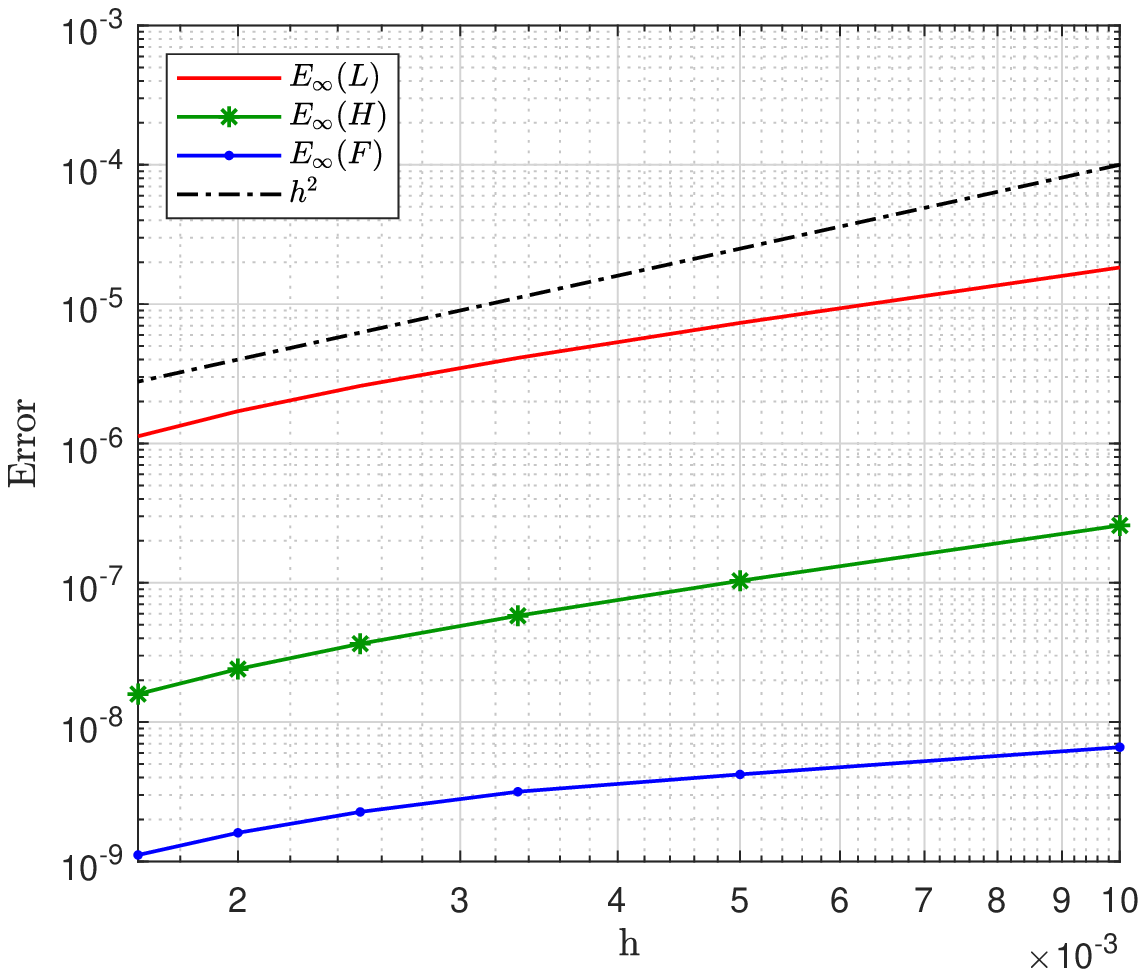}
		\captionof{figure}{\it{ The behaviour  of Maximum time-error with N=10 and various M in Log-Log scale by considering TFBM.}}\label{FTFBM}
	\end{minipage}
	\hfill
	\begin{minipage}[b]{0.45\textwidth}
		\centering
{
\begin{tabular}{lllllll}
\hline
\multicolumn{1}{c}{} & \multicolumn{1}{c}{} & \multicolumn{5}{c}{\text{Rate of convergence for}} \\ 
\cline{3-7}
\multicolumn{1}{c}{M} & \multicolumn{1}{c}{} & \multicolumn{1}{c}{L} & \multicolumn{1}{c}{} & \multicolumn{1}{c}{H} & \multicolumn{1}{c}{} & \multicolumn{1}{c}{F} \\ 
\cline{1-1}\cline{3-3}\cline{5-5}\cline{7-7}
\multicolumn{1}{c}{100} & \multicolumn{1}{c}{} & \multicolumn{1}{c}{-} & \multicolumn{1}{c}{} & \multicolumn{1}{c}{-} & \multicolumn{1}{c}{} & \multicolumn{1}{c}{-} \\ 
\multicolumn{1}{c}{200} & \multicolumn{1}{c}{} & \multicolumn{1}{c}{1.321} & \multicolumn{1}{c}{} & \multicolumn{1}{c}{1.319} & \multicolumn{1}{c}{} & \multicolumn{1}{c}{0.648} \\ 
\multicolumn{1}{c}{300} & \multicolumn{1}{c}{} & \multicolumn{1}{c}{1.426} & \multicolumn{1}{c}{} & \multicolumn{1}{c}{1.425} & \multicolumn{1}{c}{} & \multicolumn{1}{c}{0.706} \\ 
\multicolumn{1}{c}{400} & \multicolumn{1}{c}{} & \multicolumn{1}{c}{1.607} & \multicolumn{1}{c}{} & \multicolumn{1}{c}{1.607} & \multicolumn{1}{c}{} & \multicolumn{1}{c}{1.159} \\ 
\multicolumn{1}{c}{500} & \multicolumn{1}{c}{} & \multicolumn{1}{c}{1.874} & \multicolumn{1}{c}{} & \multicolumn{1}{c}{1.873} & \multicolumn{1}{c}{} & \multicolumn{1}{c}{1.546} \\ 
\multicolumn{1}{c}{600} & \multicolumn{1}{c}{} & \multicolumn{1}{c}{2.271} & \multicolumn{1}{c}{} & \multicolumn{1}{c}{2.270} & \multicolumn{1}{c}{} & \multicolumn{1}{c}{2.012} \\ 
\hline
\end{tabular}}
	\vspace{.8cm}\captionof{table}{\it{The rate of convergence with respect to time variable with N=10 and various M by considering TFBM.}}\label{NTFBM}
	\end{minipage}
\end{minipage}\\

In Table \ref{spaceerrTFBM}, we have presented the maximum space-errors  with $M=100$ and various values of $N$. To better see the space-error of numerical results, Figure \ref{report_1} is presented. 
\begin{table}
\begin{center}
{\begin{tabular}{lllllllll}
\hline
\multicolumn{1}{c}{Error} & \multicolumn{1}{c}{} & \multicolumn{1}{c}{N=2} & \multicolumn{1}{c}{} & \multicolumn{1}{c}{N=4} & \multicolumn{1}{c}{} & \multicolumn{1}{c}{N=6} & \multicolumn{1}{c}{} & \multicolumn{1}{c}{N=8} \\ 
\cline{1-1}\cline{3-3}\cline{5-5}\cline{7-7}\cline{9-9}
\multicolumn{1}{c}{L} & \multicolumn{1}{c}{} & \multicolumn{1}{c}{0.399186e-05} & \multicolumn{1}{c}{} & \multicolumn{1}{c}{0.286885e-05} & \multicolumn{1}{c}{} & \multicolumn{1}{c}{0.190745e-05} & \multicolumn{1}{c}{} & \multicolumn{1}{c}{0.119180e-05} \\ 
\multicolumn{1}{c}{H} & \multicolumn{1}{c}{} & \multicolumn{1}{c}{0.504070e-06} & \multicolumn{1}{c}{} & \multicolumn{1}{c}{0.359887-06} & \multicolumn{1}{c}{} & \multicolumn{1}{c}{0.239578e-06} & \multicolumn{1}{c}{} & \multicolumn{1}{c}{0.149847e-06} \\ 
\multicolumn{1}{c}{F} & \multicolumn{1}{c}{} & \multicolumn{1}{c}{0.293787e-07} & \multicolumn{1}{c}{} & \multicolumn{1}{c}{0.217007e-07} & \multicolumn{1}{c}{} & \multicolumn{1}{c}{0.144762e-07} & \multicolumn{1}{c}{} & \multicolumn{1}{c}{0.905871e-08} \\  
\hline
\end{tabular}}
\end{center}
\caption{\it{Maximum space-error with M=100 and various N by considering TFBM. }}
\label{spaceerrTFBM}
\end{table}
\begin{figure}[!ht]
\centering
\includegraphics[scale=.7]{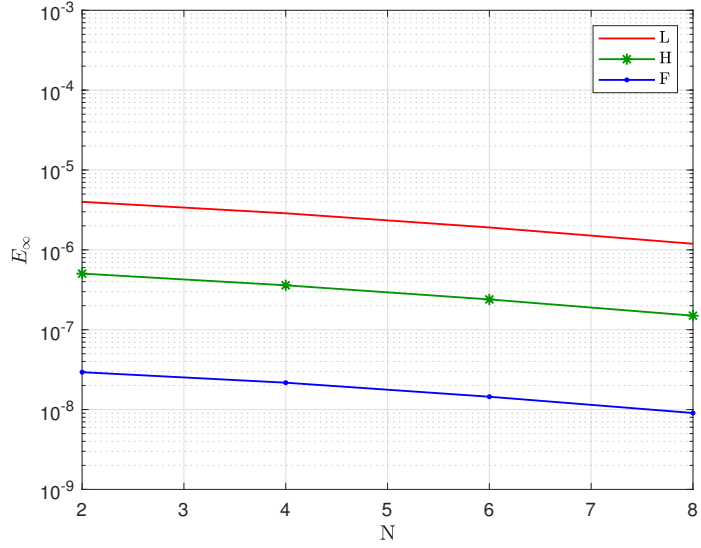}
\caption{\it{Maximum space-error functions with M=100 and various N by considering TFBM.}}
\label{report_1}
\end{figure}
\begin{table}
\begin{center}
\tiny{\begin{tabular}{lllllllllllllllllll}
\hline
CN &  & \multicolumn{1}{c}{N=4} & \multicolumn{1}{c}{} & \multicolumn{1}{c}{N=6} & \multicolumn{1}{c}{} & \multicolumn{1}{c}{N=8} & \multicolumn{1}{c}{} & \multicolumn{1}{c}{N=10} & \multicolumn{1}{c}{} & \multicolumn{1}{c}{N=12} & \multicolumn{1}{c}{} & \multicolumn{1}{c}{N=14} & \multicolumn{1}{c}{} & \multicolumn{1}{c}{N=16} & \multicolumn{1}{c}{} & \multicolumn{1}{c}{N=18} & \multicolumn{1}{c}{} & \multicolumn{1}{c}{N=20} \\ 
\cline{1-1}\cline{3-3}\cline{5-5}\cline{7-7}\cline{9-9}\cline{11-11}\cline{13-13}\cline{15-15}\cline{17-17}\cline{19-19}
Eq. (\ref{Lb}) &  & \multicolumn{1}{c}{272e+04} & \multicolumn{1}{c}{} & \multicolumn{1}{c}{115e+06} & \multicolumn{1}{c}{} & \multicolumn{1}{c}{529e+07} & \multicolumn{1}{c}{} & \multicolumn{1}{c}{258e+09} & \multicolumn{1}{c}{} & \multicolumn{1}{c}{131e+11} & \multicolumn{1}{c}{} & \multicolumn{1}{c}{689e+12} & \multicolumn{1}{c}{} & \multicolumn{1}{c}{\textbf{369e+12}} & \multicolumn{1}{c}{} & \multicolumn{1}{c}{\textbf{202e+15}} & \multicolumn{1}{c}{} & \multicolumn{1}{c}{\textbf{singular}} \\ 
Eq. (\ref{Hb}) &  & \multicolumn{1}{c}{272e+04} & \multicolumn{1}{c}{} & \multicolumn{1}{c}{115e+06} & \multicolumn{1}{c}{} & \multicolumn{1}{c}{529e+07} & \multicolumn{1}{c}{} & \multicolumn{1}{c}{258e+09} & \multicolumn{1}{c}{} & \multicolumn{1}{c}{131e+11} & \multicolumn{1}{c}{} & \multicolumn{1}{c}{689e+12} & \multicolumn{1}{c}{} & \multicolumn{1}{c}{\textbf{369e+12}} & \multicolumn{1}{c}{} & \multicolumn{1}{c}{\textbf{202e+15}} & \multicolumn{1}{c}{} & \multicolumn{1}{c}{\textbf{singular}} \\ 
Eq. (\ref{Fb}) &  & \multicolumn{1}{c}{367e+05} & \multicolumn{1}{c}{} & \multicolumn{1}{c}{303e+07} & \multicolumn{1}{c}{} & \multicolumn{1}{c}{229e+09} & \multicolumn{1}{c}{} & \multicolumn{1}{c}{166e+11} & \multicolumn{1}{c}{} & \multicolumn{1}{c}{117e+13} & \multicolumn{1}{c}{} & \multicolumn{1}{c}{814e+14} & \multicolumn{1}{c}{} & \multicolumn{1}{c}{\textbf{558e+15}} & \multicolumn{1}{c}{} & \multicolumn{1}{c}{\textbf{380e+17}} & \multicolumn{1}{c}{} & \multicolumn{1}{c}{\textbf{singular}}\\ 
\hline
\end{tabular}}
\end{center}
\caption{\it{Condition number of the coefficient matrices (CN).}}\label{condofmat}
\end{table}
It is worth to note that in the numerical results, the discrepancy between an exact value and some approximation to it, is called "maximum error" and is denoted by $E_{\infty}$.
In Table \ref{condofmat}, the condition numbers (CN) of the coefficient matrices of the collocation method are shown. In this table, high values of condition numbers are highlighted in bold. Because of the fact that the condition number of the coefficient matrices grows very fast when $N>10$, unfortunately, the Matlab software can not accurately extract the results. To overcome this difficulty, we need to propose proper trial functions which reduce the condition number significantly. In this case, we decide to consider Legendre polynomials (Jacobi polynomials with $\alpha=\beta=0$) to construct trial functions for the collocation method \cite{doha2018fully}. So, let $L_n(x)$ be the Legendre polynomial of degree $n$ and set 
\begin{equation}\label{aaaa}
p_n(x)=L_n(x)+c_nL_{n+1}(x)+d_nL_{n+2}(x), ~~ n\geq 0,
\end{equation}
where the constants $c_n$ and $d_n$ are uniquely determined in such a way that $p_n(x)$ satisfies the boundary conditions; in other word, $\dfrac{\partial p_n}{\partial x}(\pm 1)=0, ~~\forall n\geq 0$.\\
According to the features of Legendre polynomials we have
\begin{eqnarray}\nonumber
&&\dfrac{\partial p_n}{\partial x}(\pm 1)=\dfrac{\partial L_n}{\partial x}(\pm 1)+c_n\dfrac{\partial L_{n+1}}{\partial x}(\pm 1)+d_n\dfrac{\partial L_{n+2}}{\partial x}(\pm 1)=\\\label{ppp}
&& \dfrac{1}{2} (\pm 1)^{n-1} n(n+1)+c_n\dfrac{1}{2} (\pm 1)^{n} (n+1)(n+2)+d_n\dfrac{1}{2} (\pm 1)^{n+1} (n+2)(n+3)=0,
\end{eqnarray}
therefore, by solving (\ref{ppp}) one can easily conclude that
\begin{equation}
c_n=0,~~ d_n=-\dfrac{n(n+1)}{(n+2)(n+3)}.
\end{equation}
Thus (\ref{aaaa}) becomes as follows
\begin{equation}
p_n(x)=L_n(x)-\dfrac{n(n+1)}{(n+2)(n+3)}L_{n+2}(x), ~~ n\geq 0.
\end{equation}
So we can approximate the functions $L(\rho,t) ,H(\rho,t), F(\rho,t)$  in the form of (\ref{formbasis}) as follows
\begin{equation*}
L_{n+1}^N(\rho)=\sum_{i=0}^N l_i^{n+1} p_i(\rho) ,~~H_{n+1}^N(\rho)=\sum_{i=0}^N h_i^{n+1} p_i(\rho) ,~~F_{n+1}^N(\rho)=\sum_{i=0}^N f_i^{n+1} p_i(\rho),
\end{equation*}
where
\begin{equation}\label{tf}
p_i(\rho)= L_i (\rho) -\dfrac{(i+1)i}{(i+3)(i+2)}L_{i+2} (\rho) , ~~i=0,\ldots, N,
\end{equation}
which we will denote by "TFBL" (Trial Functions Based on Legendre polynomials).
\begin{remark} The scaling factors in the trial functions (\ref{tf}) play the role of precondition factor for the collocation matrices and reduce the condition number. \cite{huang2018spectral,khosravian2017generalized}.
\end{remark}
\begin{table}
\begin{center}
\scriptsize{\begin{tabular}{llllllllllllll}
\hline
{Error of} &  & \multicolumn{1}{c}{M=100} & \multicolumn{1}{c}{} & \multicolumn{1}{c}{M=200} & \multicolumn{1}{c}{} & \multicolumn{1}{c}{M=300} & \multicolumn{1}{c}{} & \multicolumn{1}{c}{M=400} & \multicolumn{1}{c}{} & \multicolumn{1}{c}{M=500} & \multicolumn{1}{c}{} & \multicolumn{1}{c}{M=600}\\ 
\cline{1-1}\cline{3-3}\cline{5-5}\cline{7-7}\cline{9-9}\cline{11-11}\cline{13-13}
\multicolumn{1}{c}{L} &  & \multicolumn{1}{c}{$2.21985e-06$} & \multicolumn{1}{c}{} & \multicolumn{1}{c}{$9.74874e-07$} & \multicolumn{1}{c}{} & \multicolumn{1}{c}{$5.66407e-07$} & \multicolumn{1}{c}{} & \multicolumn{1}{c}{$3.63391e-07$} & \multicolumn{1}{c}{} & \multicolumn{1}{c}{$2.41969e-07$} & \multicolumn{1}{c}{} & \multicolumn{1}{c}{$1.61184e-07$}\\ 
\multicolumn{1}{c}{H} &  & \multicolumn{1}{c}{$8.08703e-06$} & \multicolumn{1}{c}{} & \multicolumn{1}{c}{$3.67228e-06$} & \multicolumn{1}{c}{} & \multicolumn{1}{c}{$2.13760e-06$} & \multicolumn{1}{c}{} & \multicolumn{1}{c}{$1.37270e-06$} & \multicolumn{1}{c}{} & \multicolumn{1}{c}{$9.14550e-07$} & \multicolumn{1}{c}{} & \multicolumn{1}{c}{$6.09439e-07$}\\ 
\multicolumn{1}{c}{F} &  & \multicolumn{1}{c}{$3.89421e-06$} & \multicolumn{1}{c}{} & \multicolumn{1}{c}{$1.68696e-06$} & \multicolumn{1}{c}{} & \multicolumn{1}{c}{$9.75669e-07$} & \multicolumn{1}{c}{} & \multicolumn{1}{c}{$6.24533e-07$} & \multicolumn{1}{c}{} & \multicolumn{1}{c}{$4.15285e-07$} & \multicolumn{1}{c}{} & \multicolumn{1}{c}{$2.76382e-07$} \\ 
\hline
\end{tabular}}
\end{center}
\caption{\it{Maximum time-error with N=50 and various M by considering TFBL.}}\label{timeerrorth}
\end{table}
\begin{figure}
\centering
\includegraphics[scale=.75]{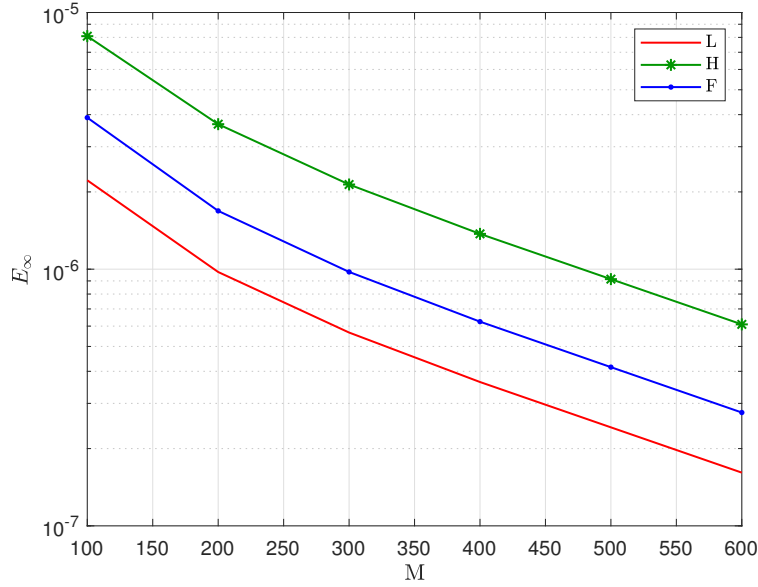}
\caption{\it{Maximum time-error with N=50 and various M by considering TFBL.}}
\label{ertime}
\end{figure}
\begin{minipage}{\textwidth}
	\begin{minipage}[b]{0.5\textwidth}
		\centering
	\includegraphics[scale=.7]{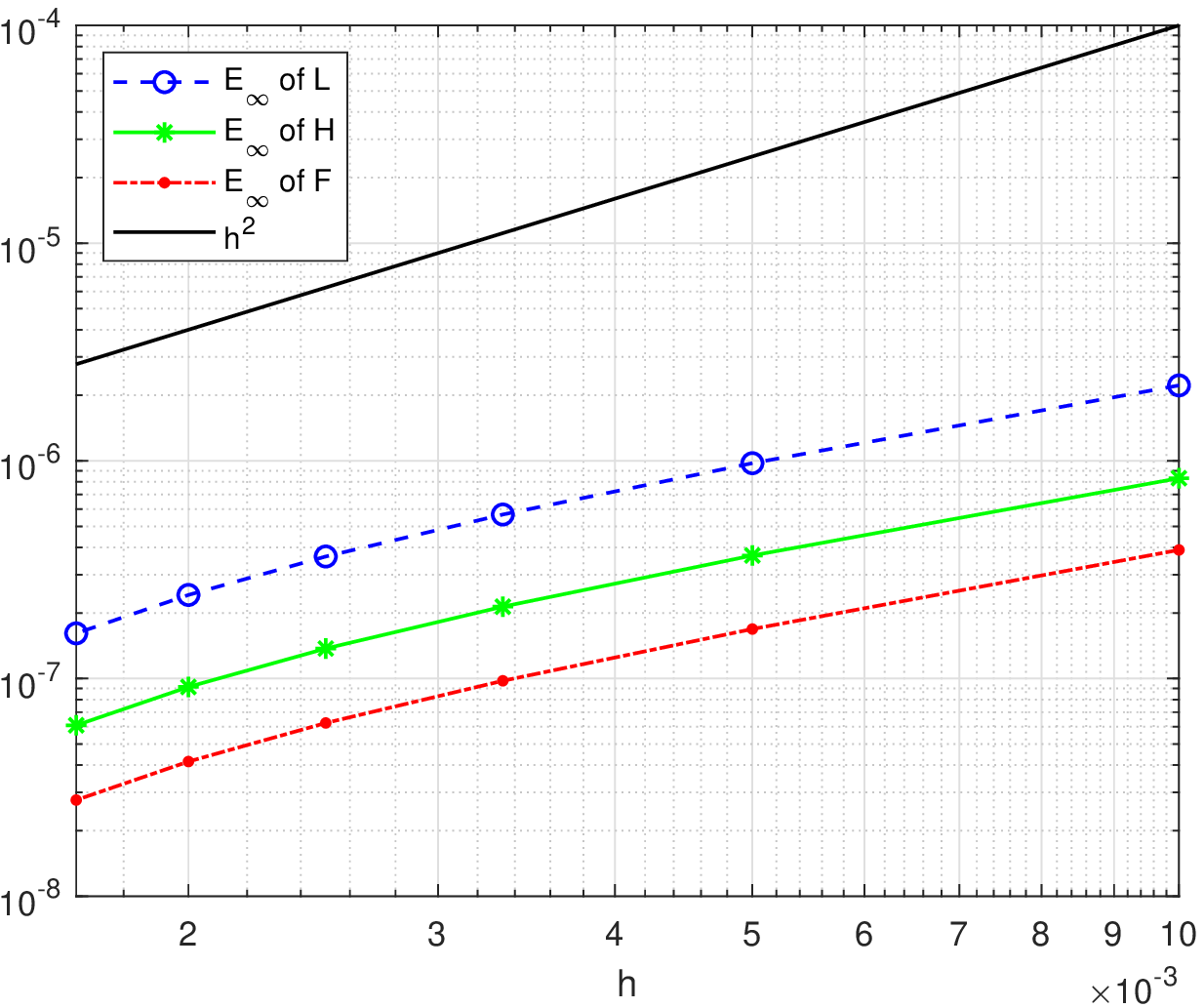}
		\captionof{figure}{\it{The behaviour of maximum time error with N=50 and various M in Log-Log scale by considering TFBL.}}\label{nesbat_orthof}
	\end{minipage}
	\hfill
	\begin{minipage}[b]{0.4\textwidth}
		\centering
{
\begin{tabular}{lllllll}
\hline
\multicolumn{1}{c}{} & \multicolumn{1}{c}{} & \multicolumn{5}{c}{\text{Rate of convergence for}} \\ 
\cline{3-7}
\multicolumn{1}{c}{M} & \multicolumn{1}{c}{} & \multicolumn{1}{c}{L} & \multicolumn{1}{c}{} & \multicolumn{1}{c}{H} & \multicolumn{1}{c}{} & \multicolumn{1}{c}{F} \\ 
\cline{1-1}\cline{3-3}\cline{5-5}\cline{7-7}
\multicolumn{1}{c}{100} & \multicolumn{1}{c}{} & \multicolumn{1}{c}{-} & \multicolumn{1}{c}{} & \multicolumn{1}{c}{-} & \multicolumn{1}{c}{} & \multicolumn{1}{c}{-} \\ 
\multicolumn{1}{c}{200} & \multicolumn{1}{c}{} & \multicolumn{1}{c}{1.118} & \multicolumn{1}{c}{} & \multicolumn{1}{c}{1.179} & \multicolumn{1}{c}{} & \multicolumn{1}{c}{1.206} \\ 
\multicolumn{1}{c}{300} & \multicolumn{1}{c}{} & \multicolumn{1}{c}{1.339} & \multicolumn{1}{c}{} & \multicolumn{1}{c}{1.334} & \multicolumn{1}{c}{} & \multicolumn{1}{c}{1.350} \\ 
\multicolumn{1}{c}{400} & \multicolumn{1}{c}{} & \multicolumn{1}{c}{1.542} & \multicolumn{1}{c}{} & \multicolumn{1}{c}{1.539} & \multicolumn{1}{c}{} & \multicolumn{1}{c}{1.550} \\ 
\multicolumn{1}{c}{500} & \multicolumn{1}{c}{} & \multicolumn{1}{c}{1.822} & \multicolumn{1}{c}{} & \multicolumn{1}{c}{1.822} & \multicolumn{1}{c}{} & \multicolumn{1}{c}{1.828} \\ 
\multicolumn{1}{c}{600} & \multicolumn{1}{c}{} & \multicolumn{1}{c}{2.228} & \multicolumn{1}{c}{} & \multicolumn{1}{c}{2.226} & \multicolumn{1}{c}{} & \multicolumn{1}{c}{2.223} \\ 
\hline
\end{tabular}}
		\vspace{.8cm}\captionof{table}{\it{ The rate of convergence with respect to the time variable with N=50 and various M by considering TFBL.}}\label{nesbat_orthot}
	\end{minipage}
\end{minipage}
\begin{table}
\begin{center}
\scriptsize{\begin{tabular}{lllllllllll}
\hline
\multicolumn{1}{c}{Error of} & \multicolumn{1}{c}{} & \multicolumn{1}{c}{N=10} & \multicolumn{1}{c}{} & \multicolumn{1}{c}{N=20} & \multicolumn{1}{c}{} & \multicolumn{1}{c}{N=40} & \multicolumn{1}{c}{} & \multicolumn{1}{c}{N=80} & \multicolumn{1}{c}{} & \multicolumn{1}{c}{N=100} \\ 
\cline{1-1}\cline{3-3}\cline{5-5}\cline{7-7}\cline{9-9}\cline{11-11}
\multicolumn{1}{c}{L} & \multicolumn{1}{c}{} & \multicolumn{1}{c}{$1.53065e-10$} & \multicolumn{1}{c}{} & \multicolumn{1}{c}{$7.04107e-12$} & \multicolumn{1}{c}{} & \multicolumn{1}{c}{$6.98306e-13$} & \multicolumn{1}{c}{} & \multicolumn{1}{c}{$1.21789e-13$} & \multicolumn{1}{c}{} & \multicolumn{1}{c}{$7.34745e-14$} \\ 
\multicolumn{1}{c}{H} & \multicolumn{1}{c}{} & \multicolumn{1}{c}{$2.96583e-11$} & \multicolumn{1}{c}{} & \multicolumn{1}{c}{$8.76799e-13$} & \multicolumn{1}{c}{} & \multicolumn{1}{c}{$2.18054e-13$} & \multicolumn{1}{c}{} & \multicolumn{1}{c}{$8.28391e-14$} & \multicolumn{1}{c}{} & \multicolumn{1}{c}{$5.54617e-14$} \\ 
\multicolumn{1}{c}{F} & \multicolumn{1}{c}{} & \multicolumn{1}{c}{$1.44016e-11$} & \multicolumn{1}{c}{} & \multicolumn{1}{c}{$6.17966e-12$} & \multicolumn{1}{c}{} & \multicolumn{1}{c}{$3.09801e-12$} & \multicolumn{1}{c}{} & \multicolumn{1}{c}{$1.69595e-12$} & \multicolumn{1}{c}{} & \multicolumn{1}{c}{$1.34437e-12$} \\ 
\hline
\end{tabular}}
\end{center}
\caption{\it{Maximum space-error with M=200 and various N by considering TFBL. }}
\label{spaceerrorth}
\end{table}
\begin{figure}
\centering
\includegraphics[scale=.8]{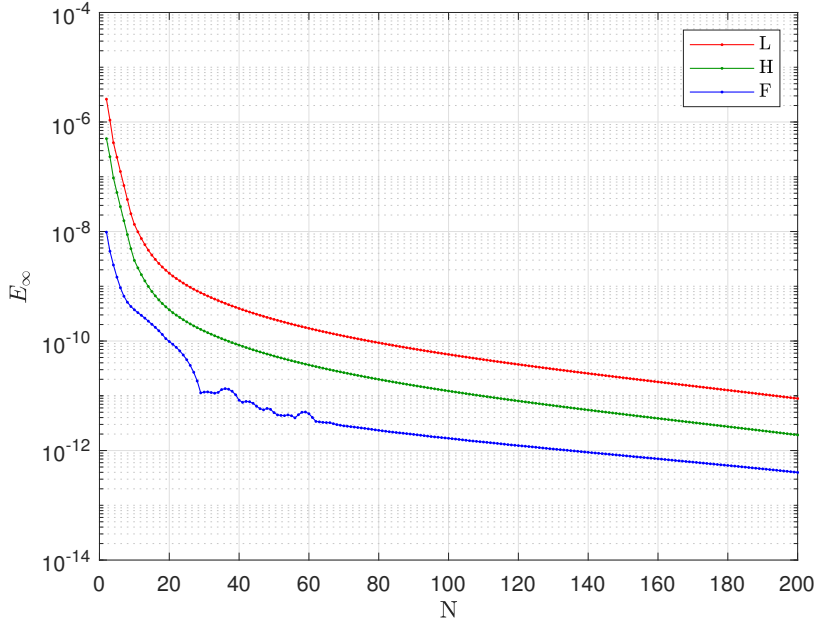}
\caption{\it{Maximum space error function for M=200 and various N.}}
\label{erspace}
\end{figure}

It should be noted that in order to use Legendre polynomials, we map the domain of the problem (\ref{fin1})-(\ref{fin2}) to $[-1,1]$.\\
We also point out that the Gauss quadrature points $\{x_i^{0,0}\}_{i=1}^{N}$ are considered as the collocation points.
As mentioned in TFBM case, because of the stability and convergence of the presented method, we consider the solution of the problem with $N=100$ and $M=1000$ as an exact solution.
In Table \ref{timeerrorth}, we have presented the maximum time-error $N=51$ and various values of $M$. It is observed that the time-error of the presented approach numerically, we report the obtained results in Figure \ref{ertime}. As we can see from Table \ref{nesbat_orthot} and Figure \ref{nesbat_orthof}, the non-classical finite difference method presented in (\ref{timed}) has almost $\mathcal{O}(h^2)$ error, which verifies our theoretical results presented in convergence Theorem \ref{convergetheo}. In Table \ref{spaceerrorth}, we illustrate the maximum space-errors $M=200$ and various values of $N$ by considering the solution of the problem with $N=200$ and $M=200$ as an exact solution. To better observe the space-error of numerical results, Figure \ref{erspace} is presented.
Also, the condition number of the coefficient matrices constructed by collocation method for equations (\ref{Lb})-(\ref{vb}) are shown in Table \ref{condt} and Figure \ref{condf}; which shows that, as we expect from TFBL, the condition number of matrices do not increase significantly by increasing $N$ even for $N=150$ . For a better comparison of condition numbers in both TFBM and TFBL cases, Figure \ref{CONCOM} is presented. One can easily see that using Legendre polynomials, the condition number of the coefficient matrices in the collocation method is significantly less than TFBM.\\
Now, in this position, the examination of  the numerical solutions from the perspective of biology and simulation is reported by presenting the rate of tumor growth with three various values of pairs $(L_0, H_0)$. The results are presented in Table \ref{radious} and Figure \ref{radius_of_plaque}, and they are compared to the risk map illustrated in figure \ref{riskmap} and 
Retrieved from \cite{friedman2015free}. As we expect, the level of $L_0$ and $H_0$ in blood directly affects  the growth and shrink of the plaque, that means for the values of $(L_0, H_0)$ below the "zero growth", the plaque grows as shown in Figure \ref{arterycom1}, and  for the values of $(L_0, H_0)$ above the "zero growth" the plaque shrinks, as shown in Figure \ref{arterycom2}. To better see the radius changes, a small part of the plaque in the vessel is magnified and is presented in the figures mentioned above. It is noteworthy that arrows in these two figures indicate the direction of growth or shrink of the plaque.
\begin{minipage}{\textwidth}
	\begin{minipage}[b]{0.40\textwidth}
		\centering
	\includegraphics[scale=.8]{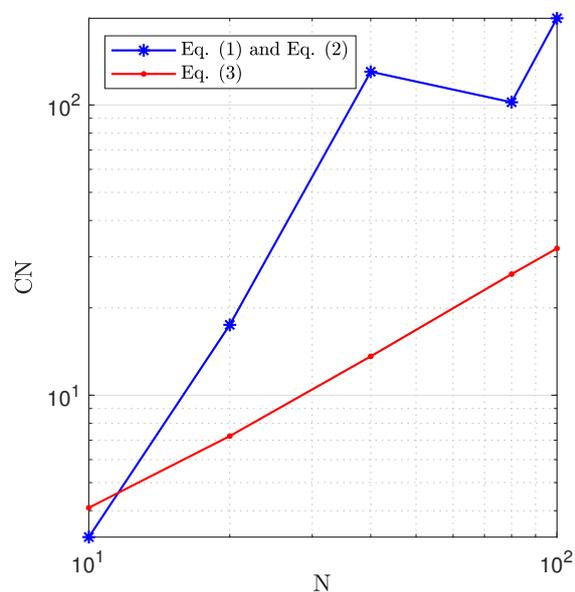}
		\captionof{figure}{\it{Condition number of the coefficient matrices (CN).}}\label{condf}
	\end{minipage}
	\hfill
	\begin{minipage}[b]{0.5\textwidth}
		\centering
{\begin{tabular}{|l|l|l|l|}
\hline
\multicolumn{1}{|c|}{N} & \multicolumn{1}{c|}{Eq. (\ref{Lb})} & \multicolumn{1}{c|}{Eq. (\ref{Hb})} & \multicolumn{1}{c|}{Eq. (\ref{Fb})} \\ 
\hline
\multicolumn{1}{|c|}{10} & \multicolumn{1}{c|}{3.245} & \multicolumn{1}{c|}{3.245} & \multicolumn{1}{c|}{4.096} \\ 
\hline
\multicolumn{1}{|c|}{20} & \multicolumn{1}{c|}{17.475} & \multicolumn{1}{c|}{17.475} & \multicolumn{1}{c|}{7.230} \\ 
\hline
\multicolumn{1}{|c|}{40} & \multicolumn{1}{c|}{1.302e+02} & \multicolumn{1}{c|}{1.302e+02} & \multicolumn{1}{c|}{13.61} \\ 
\hline
\multicolumn{1}{|c|}{80} & \multicolumn{1}{c|}{1.022e+02} & \multicolumn{1}{c|}{1.022e+02} & \multicolumn{1}{c|}{26.14} \\ 
\hline
\multicolumn{1}{|c|}{100} & \multicolumn{1}{c|}{1.990e+02} & \multicolumn{1}{c|}{1.990e+02} & \multicolumn{1}{c|}{32.05} \\ 
\hline
\end{tabular}}
		\vspace{.9cm}\captionof{table}{\it{Condition number of the coefficient matrices.}}\label{condt}
	\end{minipage}
\end{minipage}\\
\begin{figure}
\centering
\includegraphics[scale=.8]{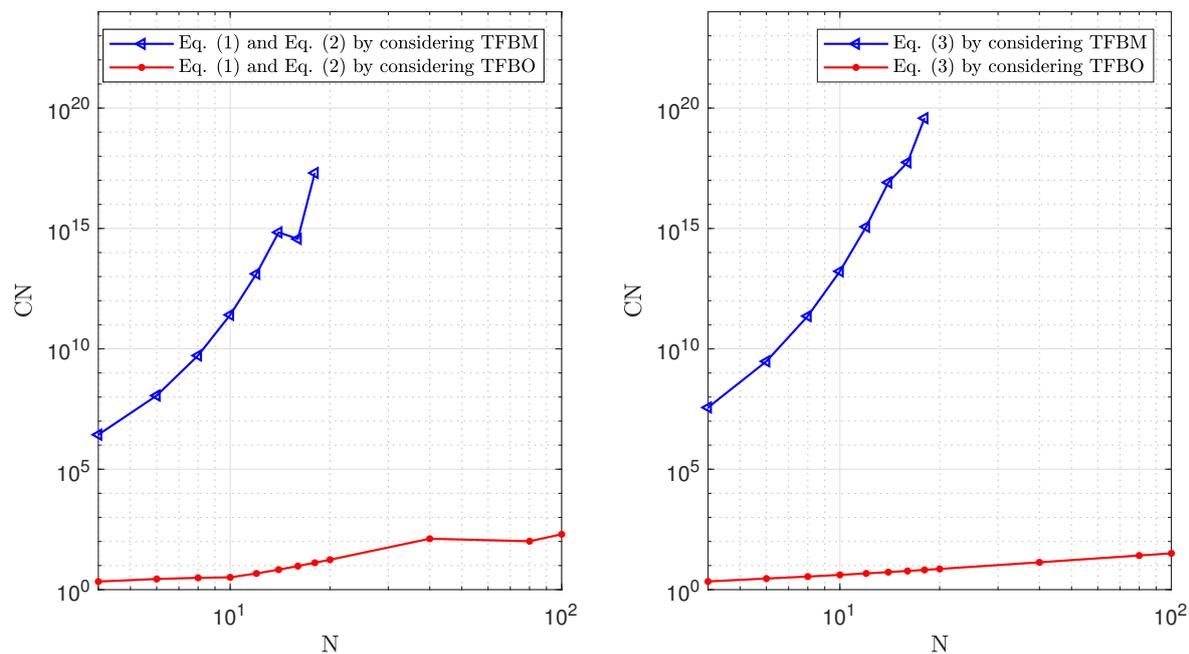}
\caption{\it{Comparison of condition numbers of the coefficient matrices in both TFBM and TFBL cases.}}
\label{CONCOM}
\end{figure}
\begin{table}
\begin{center}
\scriptsize{\begin{tabular}{lllllllllllllll}
\hline
$(L_0, H_0)$ &  & $0\,\,$ day &  & $10\,\,$ day &  & $20\,\,$ day &  & $30\,\,$ day &  & $40\,\,$ day &  & $50\,\,$ day &  & $60\,\,$ day \\ 
\cline{1-1}\cline{3-3}\cline{5-5}\cline{7-7}\cline{9-9}\cline{11-11}\cline{13-13}\cline{15-15}
$(150,45)$ &  & $0.900000$ &  & $0.897461$ &  & $0.894874$ &  & $0.892264$ &  & $ 0.889634$ &  & $0.886984$ &  & $0.884316$ \\ 
$(120,60)$ &  & $0.900000$ &  & $0.907537$ &  & $0.914481$ &  & $0.920822$ &  & $0.926622$ &  & $0.931935$ &  & $0.936810$ \\ 
$(46.5,139.5)$ &  & $0.900000$ &  & $0.900008$ &  & $0.900018$ &  & $0.900027$ &  & $0.900037$ &  & $0.900046$ &  & $0.900055$ \\ 
\hline
\end{tabular}}
\end{center}
\caption{\it{Variation of the radius of the plaque toward the various level of $L_0$ and $H_0$ in blood during the days. }}
\label{radious}
\end{table}
\begin{figure}
\centering
\includegraphics[scale=.8]{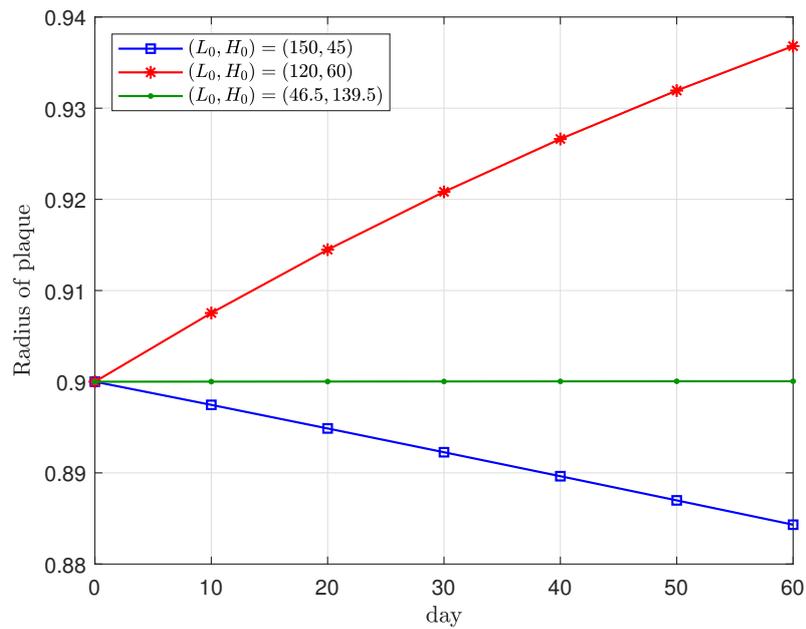}
\caption{\it{Variation of the radius of the plaque toward the various level of $L_0$ and $H_0$ in blood during the days. }}
\label{radius_of_plaque}
\end{figure}
\begin{figure}
\centering
\includegraphics[scale=.5]{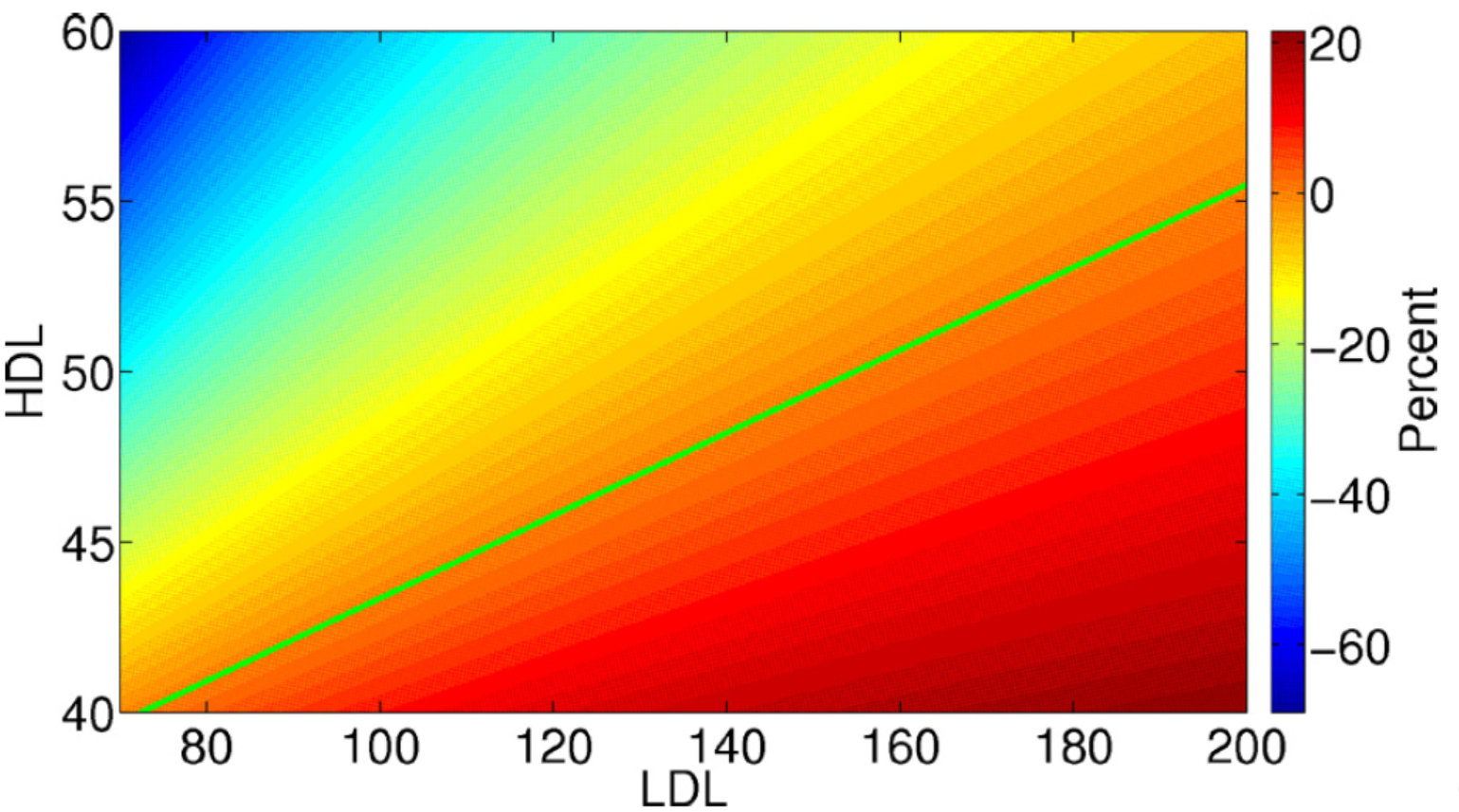}
\caption{\it{Risk Map. The values of LDL and HDL are measured in $mg/dl=10^{-4}g/cm3$ \cite{friedman2015free}.}}
\label{riskmap}
\end{figure}
\begin{figure}
\centering
\includegraphics[scale=.8]{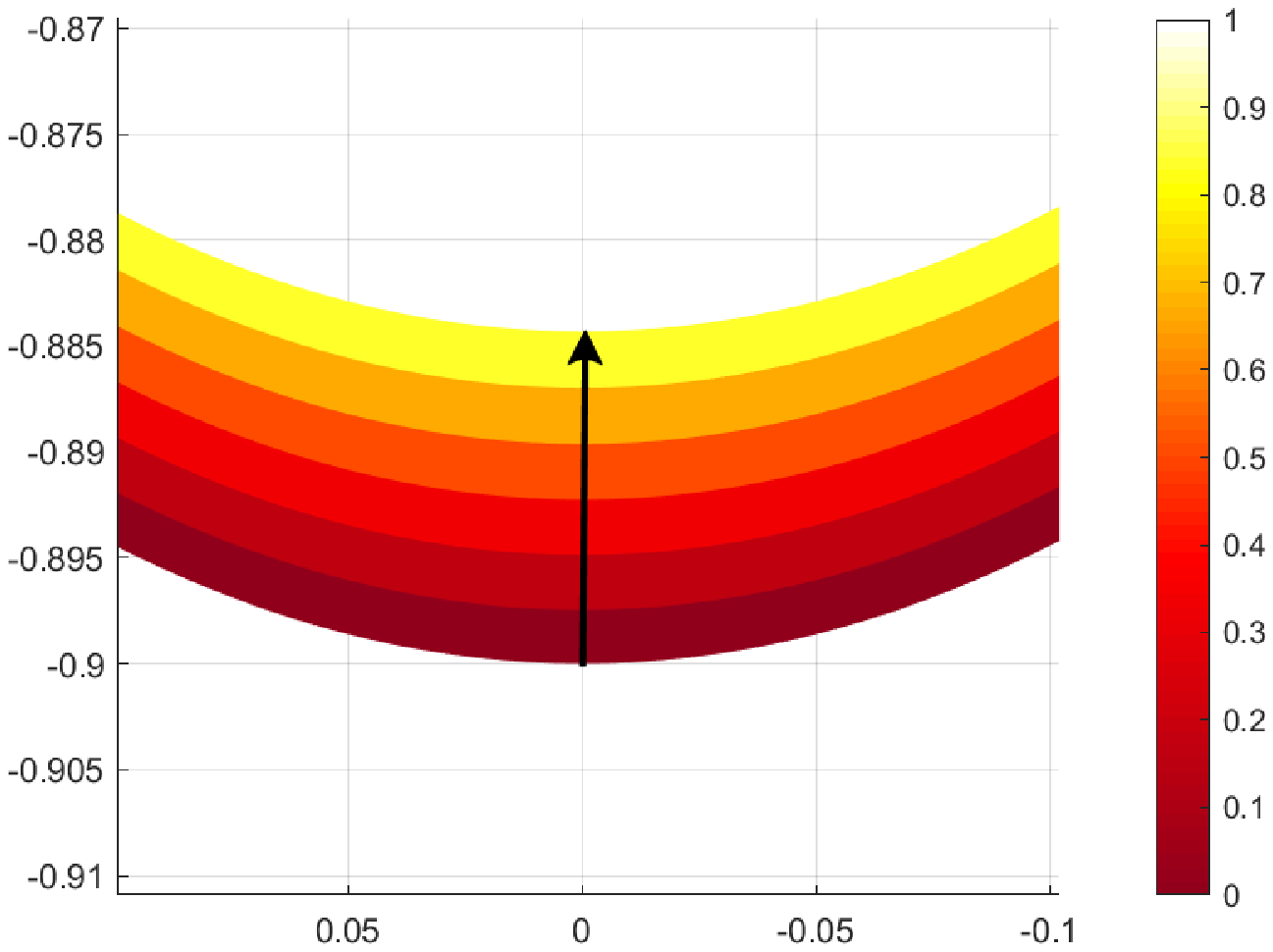}
\caption{\it{Variation of the radius of the plaque with $L_0=150 $ and $H_0=45 $ in a small part of the artery.}}
\label{arterycom1}
\end{figure}
\begin{figure}
\centering
\includegraphics[scale=.8]{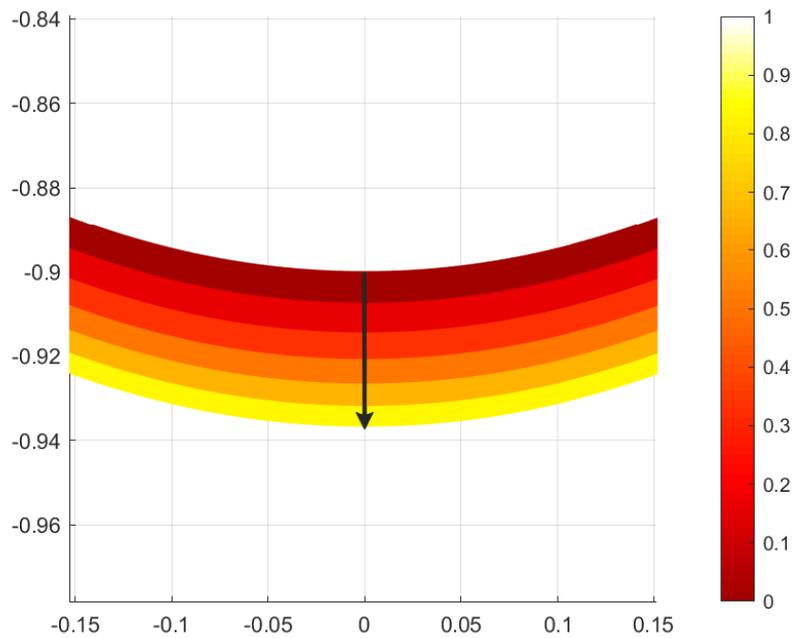}
\caption{\it{Variation of the radius of the plaque with $L_0=120 $ and $H_0=60 $ in a small part of the artery.}}
\label{arterycom2}
\end{figure}
\newpage
\section{Conclusion}
There are many mathematical methods for solving biological models. However, mathematical modeling often produces nonlinear differential equations. Therefore, we cannot always obtain the exact solution of these equations; so, developing numerical techniques to solve these equations is required. In this study, a mathematical model of Atherosclerosis is solved numerically and the convergence and stability analysis are presented. For the reader’s convenience, we give the main contributions of this study as follows \\
$\bullet$ In this article, we use the front fixing method to convert the moving boundary problem (\ref{e1})-(\ref{e3}) to a fix one (\ref{eqf1})-(\ref{eqfe}), because classical numerical methods are not effective to solve free and moving boundary problems and moreover, because of the suitability of the front fixing method to apply to problems with regular geometries along with the mesh-based methods.\\
$\bullet$ 
To achieve more comfortable results for numerical analysis, we have simplified the model by changing the mix boundary condition  of the equations (\ref{eqf1})-(\ref{eqf2}) to a Neumann one by applying a suitable change of variables (\ref{vc1})-(\ref{vc3}). \\
$\bullet$ To solve nonlinear systems, one way is to linearize the equations of the system and then, solve the linear one by classical methods. Since the model studied in this article is nonlinear, we have used Taylor theorem simultaneously both to linearize the equations and for constructing new second-order non-classical discretization formula to approximate time discretization (Finite difference method).\\
$\bullet$ In this article, we use spectral collocation method in space. To construct trial functions which satisfy the boundary conditions, the first way that comes to mind is to use a linear combination of monomial polynomials. But, there are some  problems to use this kind of polynomials. some of these problems are as follows\\
$~~~~$1. The condition number of collocation matrices increases significantly by increasing the size of the matrices for $N\geq 10$ and because of that, we are limited to increase the collocation points and achieve arbitrary errors. (See Table \ref{condofmat} and Table \ref{spaceerrTFBM})\\
$~~~~$2. The high error of collocation method affects the error of the finite difference method (See Table \ref{timeerrTFBM}).\\
To the above reasons, we use a linear combination of classical orthogonal polynomials or orthogonal functions to construct trial functions.\\
 So, using orthogonal polynomials, the condition number of collocation matrices does not increase significantly by increasing the collocation points even for $N=150$ (See Table \ref{condt}, Figure \ref{condf} and Figure \ref{CONCOM}), and because of that, the error of the collocation method decreases compared to the TFBM case (See Table \ref{spaceerrorth} and Figure \ref{erspace}).\\
$\bullet$ Moreover, the convergence and stability of the presented method is proved (See Theorem \ref{convergetheo} and Theorem \ref{stabtheo}) and the order of convergence is presented. Numerical results in both TFBM and TFBL cases show that the finite difference method displays an $\mathcal{O}(h^2)$ order of convergence, as we expect from convergence Theorem \ref{convergetheo} (See Figure \ref{FTFBM} and Figure \ref{nesbat_orthof}), and the space-error shows that using the collocation method, the results are converging to the exact solution.\\
$\bullet$ As illustrated in Figure \ref{radius_of_plaque} and Table \ref{radious}, we present the effect of the level of $L_0$ and $H_0$ in blood on the growth and shrink of the plaque. It is easily can be seen that for $(L_0, H_0)=(150, 45)$ (below the "zero growth"), the plaque grows and  for $(L_0, H_0)=(120, 60)$ (above the "zero growth") the plaque shrinks.

\end{document}